\pdfoutput=1

\documentclass[12pt]{article}
\usepackage[utf8]{inputenc}
\usepackage[english]{babel}
 \usepackage{amsmath}
\usepackage{amsthm}
\usepackage{amssymb}
\usepackage{hyperref, enumitem}
\usepackage{braket}
\usepackage{xcolor}
\usepackage{comment}
\usepackage{graphicx}
\usepackage[margin=2.5cm]{geometry}
 
 \usepackage[rightcaption]{sidecap}

\makeatletter
\newcommand{\mylabel}[2]{#2\def\@currentlabel{#2}\label{#1}}
\makeatother

\newcommand{\epsi}{\varepsilon}
\newcommand{\E}{{\mathrm{e}}}
\newcommand{\I}{\mathrm{i}}
 \newcommand{\R}{ \mathbb{R} }
  \newcommand{\Sph}{ \mathbb{S} }
\newcommand{\C}{ \mathbb{C} }
\newcommand{\N}{ \mathbb{N} }
\newcommand{\Z}{ \mathbb{Z} }
\newcommand{\D}{\mathrm{d}}
\newcommand{\T}{{\mathbb{T}}}

          \newcommand{\K}{{\boldsymbol{k}}}
                 \newcommand{\cC}{{\boldsymbol{c}}}
                    \newcommand{\omf}{{\boldsymbol{\omega}}}
\begin{document}
	\allowdisplaybreaks
	\numberwithin{equation}{section}
	
\theoremstyle{plain}
\newtheorem{prop}{Proposition}[section]
\newtheorem{thm}[prop]{Theorem}
\newtheorem{lem}[prop]{Lemma}
\newtheorem{cor}[prop]{Corollary}
\newtheorem{theorem}{Theorem}
\renewcommand*{\thetheorem}{\Roman{theorem}}
\newtheorem*{main}{Main Results}

\theoremstyle{definition}
\newtheorem{defi}[prop]{Definition}
\newtheorem{rmk}[prop]{Remark}
	
\setcounter{tocdepth}{1}

\title{Deformational rigidity of integrable metrics on the torus}
\author{Joscha Henheik\footnote{IST Austria, Am Campus 1, A-3400 Klosterneuburg\newline Mail: \href{mailto:joscha.henheik@ist.ac.at}{joscha.henheik@ist.ac.at} }
}
\maketitle

\begin{abstract}
	It is conjectured that the only integrable metrics on the two-dimensional torus are Liouville metrics. In this paper, we study a deformative version of this conjecture: We consider integrable deformations of a non-flat Liouville metric in a conformal class and show that for a fairly large class of such deformations the deformed metric is again Liouville. Since our method of proof immediately carries over to higher dimensional tori, we obtain analogous statements in this more general case. In order to put our results in perspective, we review existing results about integrable metrics on the torus.
	\\[2mm]
	\textbf{AMS Classes (2020):} 37J35, 37J39, 53D25, 37C83, 37J40, 70H06. 
\end{abstract}

\section{Introduction} \label{sec:Introduction}
Let $\T^2 = \R^2 / \Z^2$ be the two-dimensional torus being equipped with a $C^2$-smooth global Liouville metric $g$, i.e.~having line element
\begin{equation} \label{eq:lioumet}
\D s^2 = \big(f_1(x^1) + f_2(x^2)\big) \, \big((\D x^1)^2 + (\D x^2)^2\big)\,,
\end{equation}  
where $(x^1,x^2) \in \T^2$ are the standard periodic coordinates and $f_1, f_2 \in C^2(\T)$ are positive Morse functions\footnote{Recall that Morse functions on a manifold $M$ are characterized by having no degenerate critical points. They form a dense and open set in $C^2(M)$ and are thus `generic'.} or positive constants and thus `non-degenerate'. The corresponding geodesic flow (see Section \ref{subsec:geoflow}) is well known to be integrable and a longstanding folklore conjecture says that \emph{Liouville metrics are the only integrable metrics on $\T^2$}. {We emphasize that, in this context, integrability always allows for singularities in the foliation of the phase space of the naturally associated Hamiltonian system, which is made precise in Definition \ref{def:integrable} below.}
	
Although the validity of the folklore conjecture appeared conceivable for a long time, there is strong indication for it being false in its very general form, as shown in \cite{corsikaloshin}, where the authors constructed a counterexample which is locally integrable in a $p$-cone in the cotangent bundle. However, certain suitably weakened conjectures are still believed to be true, which is supported by a variety of partial results obtained in this direction, starting from classical ones by Dini \cite{dini}, Darboux \cite{darboux}, and Birkhoff \cite{birkhoff1927}  and further developed in \cite{babenkoneko1995,kiyohara1991,kolokoltsov}. In particular, several works by Bialy, Mironov \cite{bialy1987, bialymironov1, bialymironov2, bialymironov3}, Denisova, Kozlov, Treshchev \cite{denisovakozlov1, denisovakozlov2, denisovakozlov3, kozlovtreshchev, denisovakozlovtreshev2012}, Mironov \cite{mironov}, and others \cite{babenkoneko1995, kolokoltsov, agapov, taimanov}, strongly indicate the validity of the following (yet unproven) conjecture:\footnote{See \cite{probl1, probl2} for recent surveys on open problems and questions concerning geodesics and integrability of finite-dimensional systems.} \emph{Every polynomially integrable metric $g$ on $\T^2$ is of Liouville type.} We refer to Section \ref{sec:integrablemetricsonT2} for details. 

In this paper, we are concerned with a \emph{perturbative} version of the folklore conjecture: Let $(g_\epsi)_{|\epsi| \le \epsi_0}$ for some small $\epsi_0 > 0$ be a family of perturbations of $g \equiv g_0$ in the same conformal class\footnote{Note that on the torus there exist global isothermal coordinates \cite[Chapter~11]{bolsinovfomenko}.} having line-element
\begin{equation} \label{eq:pertlioumet}
\D s^2_\epsi = \big(f_1(x^1) + f_2(x^2) + \epsi \lambda(x^1, x^2)\big) \, \big((\D x^1)^2 + (\D x^2)^2\big)\,,
\end{equation}  
where $\lambda \in C^2(\T^2)$ is assumed to be a Morse function (or constant) and have an absolutely convergent Fourier series. We will assume that the perturbed family $g_\epsi$ \emph{remains integrable}, meaning that there are two independent $C^2$-smooth first integrals and outside of a hypersurface the phase space is foliated by invariant tori. In particular, this means that the deformation preserves sufficiently many rational invariant tori. Then we obtain that $\lambda$ is necessarily \emph{separable} in a sum of two single-valued functions, i.e.
\begin{equation*} \label{eq:separable}
\lambda(x^1,x^2) = \lambda_1(x^1) + \lambda_2(x^2)
\end{equation*}
for some $\lambda_1, \lambda_2 \in C^2(\T)$. Therefore, our main results formulated (somewhat informally) below assert the following: 
\begin{quote} \it 
The class of Liouville metrics is deformationally rigid under a fairly wide class of integrable conformal perturbations.
\end{quote}
{To the best of our knowledge, this is the first instance of a rigidity result for (not necessarily analytically) integrable dynamical systems allowing for singularities in the invariant foliation of the unperturbed system.} The precise statements of our main results are given in Theorem~\ref{thm:1}, Theorem~\ref{thm:2} and Theorem~\ref{thm:3} in Section \ref{sec:Mainresults}. 

\begin{main}
Let $g$ be a non-degenerate Liouville metric on $\T^2$ as in \eqref{eq:lioumet} and assume that the family $(g_\epsi)_{|\epsi| \le \epsi_0}$ of perturbations defined in \eqref{eq:pertlioumet} remains integrable. Then we have the following: 
\begin{itemize}
\item[(i)] In case that $f_1, f_2 \equiv \mathrm{const}.$, then $\lambda$ is separable. 
\item[(ii)] In case that $f_1 \equiv \mathrm{const}.$, $\lambda$ is a trigonometric polynomial in $x^2$ and the relative difference between $f_2$ and its mean value is small, then $\lambda$ is separable. 

If, additionally, $f_2$ is analytic, we have that $\lambda$ is separable, irrespective of the size $\mu_2$ of the fluctuations of $f_2$ (but only for $\mu_2$ outside of an exceptional (Lebesgue) null-set). 
\item[(iii)] In general, if $\lambda$ is a trigonometric polynomial and the relative differences between $f_1, f_2$ and their mean values are small, then $\lambda$ is separable. 

If, additionally, $f_i$ is analytic (for one or both $i = 1, 2$), we have that $\lambda$ is separable, irrespective of the size $\mu_i$ of the fluctuations of $f_i$ (outside of an exceptional null-set). 
\end{itemize}
\end{main}
It is straightforward to generalize our results to higher dimensional tori $\T^d = \R^d / \Z^d$. In order to ease notation and make the presentation clearer, we only mention it here and postpone a more detailed discussion to Appendix \ref{app:higher dim}. 
\begin{rmk}{\rm (Generalization to higher dimensions)} \label{rmk:higher dim} \\
Analogously to \eqref{eq:lioumet}, let $\T^d$ be equipped with a $C^2$-smooth global Liouville metric $g$ having line element
\begin{equation} \label{eq:lioumet d}
\D s^2 = \big(  f_1(x^1) + \ldots + f_d(x^d) \big) \, \big( (\D x^1)^2 + \ldots + (\D x^d)^2\big)\,, 
\end{equation}
where $x = (x^1, ... , x^d) \in \T^d$ are standard periodic coordinates and $f_i \in C^2 (\T)$ for $1 \le i \le d$ are positive Morse functions or constants. Again, it is easy to see that the geodesic flow is integrable. Just as in \eqref{eq:pertlioumet}, we now perturb \eqref{eq:lioumet d} in the same conformal class by some $\lambda \in C^2(\T^d)$ having an absolutely convergent Fourier series. 

Under the assumption that the family of perturbed metrics $(g_\epsi)_{|\epsi| \le \epsi_0}$ remains integrable, we have the following (somewhat informal) rigidity result: \\[2mm]
\emph{Let $f_i \equiv \mathrm{const}.$ for the first $0 \le d_{\rm flat} \le d$ indices, and $f_j$ be analytic for the last $0 \le  d_{\rm anlyt} \le d - d_{\rm flat}$ indices. Then, if $\lambda$ is a trigonometric polynomial in $x^k$ for $k \in \set{d_{\rm flat} + 1, ... , d}$, and the relative differences between $f_{d_{\rm flat} + 1}, ... , f_{d - d_{\rm anlyt}}$ and their mean values are small, we have that $\lambda$ is separable, irrespective of size $\mu_j$ of the fluctuations of $f_j$ (outside of a null-set).     }\\[2mm]
This results unifies and generalizes the three separate statements given above. A precise formulation is given in Theorem \ref{thm:4} in Appendix \ref{app:higher dim}. 
\end{rmk}
The present paper is not the first study on \emph{rigidity} of important integrable systems: In \cite{ADK, KS}, Avila-Kaloshin-de Simoi and Kaloshin-Sorrentino recently solved both, a deformative and a perturbative version of the famous \emph{Birkhoff conjecture} concerning integrable billiards in two dimensions. In a nutshell, their result says that \emph{a strictly convex domain with integrable billiard dynamics sufficiently close to an ellipse is necessarily an ellipse}. This can be viewed as an analogue of the perturbative version of the folklore conjecture formulated above \cite{KS2}. More precisely, our main results concerning general $f_i \in C^2(\T)$ are similar -- in spirit -- to the deformational rigidity for ellipses of \emph{small} eccentricity (cf.~the functions $f_1,f_2$ in \eqref{eq:lioumet} having small fluctuations), which has been shown first in \cite{ADK}, later extended by Huang-Kaloshin-Sorrentino \cite{HKS} to a local notion of integrability and finally improved in \cite{Koval}. The overall strategy pursued in \cite{ADK, KS, HKS} also inspired the arguments employed in the present paper.

In a more recent work, Arnaud-Massetti-Sorrentino \cite{ArnaudMassSorr2022} (replacing the earlier preprint \cite{MassSorr2020}) studied the rigidity of integrable symplectic twist maps on the $2d$-dimensional annulus $\T^d \times \R^d$. More precisely, they consider one-parameter families $(f_\epsi)_{\epsi \in \R}$ of symplectic twist maps $f_\epsi(x,p) = f_0(x , p + \epsi \nabla G(x))$ and prove two main rigidity results: First, in the analytic category for $f_0$ and the perturbation $G$, if a \emph{single} rational invariant Lagrangian graph of $f_\epsi$ exists for infinitely many values of $\epsi$ (e.g.~an interval around zero), then $G$ must necessarily be constant. Second, if $f_0$ is analytic and \emph{completely integrable} (i.e.~not plagued with singularities in the invariant foliation of the phase space, see \cite{BialyMcKay, Suris}), $G$ is of class $C^2$, and sufficiently (infinitely) many rational invariant Lagrangian graphs of $f_\epsi$ persist for small $\epsi \neq 0$, then $G$ must necessarily be constant. Note that in this second result, the entire phase space is foliated by invariant tori, and the perturbation solely depends on the angle variables of the dynamical system.   
In this sense, Theorem~\ref{thm:1} can -- morally -- be viewed as a special case of the second result in~\cite{ArnaudMassSorr2022} (see also \cite[Theorem~2]{MassSorr2020}), but Theorem \ref{thm:2} and Theorem~\ref{thm:3} generalize this statement to more general functional dependencies of the perturbation. Apart from this, our general results (i.e.~those not concerning analytic functions $f_i$) do not require any regularity beyond the standard~$C^2$. 

As mentioned above, by assuming that the family of metrics $(g_\epsi)_{|\epsi| \le \epsi_0}$ remains integrable, we mean that, in particular, sufficiently many rational invariant tori in an isoenergy manifold of the Hamiltonians associated to the metric by the \emph{Maupertuis principle} (see Section \ref{subsec:maupertuis}) are preserved. This will be made precise in  Assumption \eqref{itm:P} below. As we will show, the preservation of an $(n,m)$-rational invariant torus `annihilates' the Fourier coefficients $\lambda_{k_1, k_2}$ with indices $(k_1, k_2) \in \{ (n,m)\}^{\perp}$ of 
\begin{equation*}
	\lambda(x,y) = \sum_{ (k_1, k_2) \in \Z^2} \lambda_{k_1, k_2} \E^{2 \pi \I (k_1 x + k_2 y)}\,,
\end{equation*}
or of the corresponding perturbing mechanical potential, denoted by $U$ later on.  We already noted that, contrary to items (ii) and (iii), the unperturbed metric in our first result is guaranteed to be \emph{completely integrable}. Moreover, the perturbation $\lambda$ depends solely on the \emph{angular} but not the \emph{action} coordinates of the unperturbed problem (see Theorem \ref{thm:liouvillearnold}).  Although the analog of this result for symplectic twist maps in this peculiar setting has already been shown in \cite{ArnaudMassSorr2022, MassSorr2020} by methods similar to ours, we reprove it by pursuing an only slightly different but original strategy, which is suitable for certain inevitable modifications for the proofs of the more general statements under item (ii) and (iii). These two cases (corresponding to surfaces of revolution and general Liouville metrics, see Section~\ref{sec:integrablemetricsonT2}) build on perturbative estimates for (possibly infinitely many) systems of linear equations for the Fourier coefficients. These are obtained from the first order term of an expansion in $\epsi$, somewhat similar to the \emph{(subharmonic) Melnikov potential} in the Poincarè-Melnikov method \cite{Guckenheimer, TreschZube, AKN}. Establishing this expansion as well as proving that the resulting systems of linear equations are of full rank requires perturbative estimates on action-angle coordinates and certain basic objects from weak KAM theory \cite{SorrentinoLecNotes}. Finally, the extension of our results for analytic functions $f_i$ beyond the perturbative regime are proven by exploiting the analytic dependence of the linear system on the size $\mu_i$ of the fluctuations of~$f_i$ (see Appendix~\ref{app:AApendulum}).  

 In the remainder of this introduction, we recall basic notions in geometry and dynamical systems, which are frequently used in this paper, and introduce the problem of classifying integrable metrics on Riemannian manifolds, in particular the torus $\T^2$, as formulated in Questions \eqref{itm:Q1} and \eqref{itm:Q2} below. The experienced reader can skip these parts in their entirety. In Section \ref{sec:Mainresults} we formulate our main results in Theorem~\ref{thm:1}, \ref{thm:2}, and \ref{thm:3}. In Section \ref{sec:integrablemetricsonT2} we present related existing results and known partial answers on the classification problem for integrable metrics on the torus $\T^2$ (a few of which have already been mentioned above) in order to put our results into context. In Section \ref{sec:proofs} we give the proofs of our main results, and, finally, comment on  possible generalizations, different approaches and a list of open problems in Section~\ref{sec:outlook}. As already mentioned above, the precise formulation of our result for higher dimensions is given in Theorem~\ref{thm:4} in Appendix \ref{app:higher dim}. A fundamental perturbation theoretic lemma on action-angle coordinates, a concise study on important analyticity properties of these, and a brief overview of the relevant aspects of weak KAM theory are presented in three further appendices.

\subsection{Geodesic flow on Riemannian manifolds} \label{subsec:geoflow}
Let $(M,g)$ be a (compact) $C^2$-smooth $n$-dimensional connected Riemannian manifold without boundary equipped with the  Riemannian metric $g = (g_{ij}(x))_{ij}$. Geodesics of the given metric $g$ are defined as smooth parameterized curves $\gamma(t) = (x^1(t), \dots, x^n(t))$
that are solutions to the system of differential equations
\begin{equation} \label{eq:geodesicequation1}
\nabla_{\dot{\gamma}} \dot{\gamma} = 0\,,
\end{equation}
where $\dot{\gamma} = \frac{\D \gamma}{\D t}$ denotes the velocity vector of the curve $\gamma$, and $\nabla$ is the covariant derivative operator related to the Levi-Civita connection associated with the metric $g$.

It is well known that for every point $x \in M$ and for every tangent vector $v \in T_xM$ there exists a unique geodesic $\gamma$ with $\gamma(0) = x$ and $\dot{\gamma}(0) = v$, which allows to define the \textit{geodesic flow} as a local $\R$-action on the tangent bundle $TM$ via
\begin{equation*} 
 \R \ni t \mapsto  \Psi_t(V) = \dot{\gamma}_V(t) \in TM \,,
\end{equation*}
where $\gamma_V$ denotes the geodesic with initial data $\dot{\gamma}_V(0)=V \in TM$. 

The geodesic equation \eqref{eq:geodesicequation1} 
can also be viewed as a Hamiltonian system on the cotangent bundle $T^*M$, and the geodesics $\gamma$ themselves can be regarded as projections of trajectories of the Hamiltonian system onto $M$. Therefore, let $x$ and $p$ be natural coordinates on the cotangent bundle $T^*M$, where $x = (x^1, \dots, x^n)$ are the coordinates of a point in $M$ (position space), and $p = (p_1,\dots, p_n)$ are the coordinates of a covector from the cotangent space $T^*_xM$ (momentum space) in the basis $\D x^1, \dots ,\D x^n$. Let $\omega = \D x \wedge \D p$ on $T^*M$ denote the standard symplectic structure and define the Hamiltonian function $H \in C^2(T^*M)$ as 
\begin{equation} \label{eq:Hamiltonian1}
H(x,p) = \frac{1}{2} \sum_{ij} g^{ij}(x)p_ip_j = \frac{1}{2} \vert p \vert_g^2\,.
\end{equation}
The related Hamiltonian vector field $X_H$, defined via $\omega(X_H, \cdot ) = \D H$, governs the associated \textit{Hamiltonian flow} $\Phi_t^{X_H}$ as a local $\R$-action on the cotangent bundle $T^*M$. A trajectory $(x(t),p(t))$ is an integral curve for the Hamiltonian vector field, if and only if the Hamiltonian system of differential equations
\begin{equation} \label{eq:Hamiltonianequations}
\begin{cases}
\dot{p}_i =& - \frac{\partial H}{\partial x^i}  \\
\dot{x}^i =& \frac{\partial H}{\partial p_i} 
\end{cases}\,,
\end{equation}
written in local coordinates, is satisfied. The Hamiltonian flow is also called a \textit{cogeodesic flow} for this special case of a Hamiltonian function \eqref{eq:Hamiltonian1}, and the geodesic flow and the cogeodesic flow are equivalent in the following sense.

\begin{prop}{\rm (Geodesic flow and cogeodesic flow, Prop.~11.1 in \cite{bolsinovfomenko})} \label{prop:cogeo} 
	\begin{itemize}
\item[(a)] If $ (x(t),p(t))$ is an integral curve for $X_H$ on $T^*M$, then the curve $x(t)$ in $M$ is a geodesic and its velocity vector $\dot{x}(t)$ satisfies $\dot{x}^i(t) = \sum_j g^{ij}(x)p_j(t)$.
\item[(b)] Conversely, if $x(t)$ is a geodesic in $M$, then the trajectory $(x(t),p(t))$, where $p_i(t) = \sum_j g_{ij} \dot{x}^j(t)$,
is an integral curve for $X_H$ on $T^*M$.
	\end{itemize}
\end{prop}
\subsection{Integrable Hamiltonian systems} \label{subsec:intriemmet}
It is natural to ask for a classification of Riemannian manifolds $(M,g)$, for which the geodesic equations \eqref{eq:geodesicequation1} can be solved explicitly. In the language of integrability of Hamiltonian systems and using the equivalence between geodesic flow and cogeodesic flow from Proposition~\ref{prop:cogeo}, we can formulate the following questions:  
\begin{itemize}
	\item[\bf(\mylabel{itm:Q1}{Q1})] \label{itm:Q1}On which manifolds exist Riemannian metrics whose (co-)geodesic flow is integrable?
	\item[\bf(\mylabel{itm:Q2}{Q2})] \label{itm:Q2}Given such a manifold, how to characterize the class of metrics with integrable geodesic flow?
\end{itemize}  
Clearly, the answers and their complexity hinge on the notion of integrability for the Hamiltonian system (see Section \ref{sec:integrablemetricsonT2}). In this paper we will be concerned with the standard notion, that is \textit{Liouville integrability}, which we recall for the readers convenience. 
\begin{defi} \label{def:integrable}
The geodesic flow on $(M,g)$ is called {\it Liouville integrable}, if there exist $n$ functions $F_1,...,F_n \in C^2(T^*M)$ (called \textit{first integrals}), that are 
\begin{itemize}
\item[(i)] functionally independent on $T^*M$, i.e.~the vector fields $X_{F_1}(x,p), ... , X_{F_n}(x,p)$ are linear independent in $T_{(x,p)}(T^*M)$ for all  $(x,p) \in \mathcal{M} \subset T^*M$, where $\mathcal{M}$ is some open and everywhere dense set of full measure (cf.~the restriction to Morse functions);
\item[(ii)] pairwise in involution, i.e.
\begin{equation*}
\{F_k,F_l\} := \omega(X_{F_k},X_{F_l}) =  \sum_{i} \left(\frac{\partial F_k}{\partial x^i} \frac{\partial F_l}{\partial p_i} - \frac{\partial F_k}{\partial p_i} \frac{\partial F_l}{\partial x^i}\right) = 0\,.
\end{equation*}
\end{itemize}
 Whenever the geodesic flow on $(M,g)$ is Liouville integrable, we call $g$ an {\it integrable metric} on $M$. Moreover, we call the Hamiltonian system \eqref{eq:Hamiltonianequations} (or the corresponding Hamiltonian~\eqref{eq:Hamiltonian1} itself) {\it integrable}, whenever the associated metric $g$ is integrable on~$M$. 
\end{defi}
\begin{rmk}
Whenever the first integrals $F_1, ... , F_n$ can be chosen to be functions that are polynomially in the momentum variables, the metric is often called \textit{polynomially integrable} or \textit{algebraically integrable}. If we aim at indicating the order of the polynomial, we speak of \textit{linearly/quadratically/... integrable} metrics. 
\end{rmk}
\begin{rmk}
Note that, since one can always choose $H=F_1$ as a first integral for the geodesic flow, the question of integrability for one-dimensional manifolds is completely answered. Therefore, the simplest manifolds, for which the answers to \eqref{itm:Q1} and \eqref{itm:Q2} are non-trivial, are two-dimensional. 
\end{rmk}
In this work, we are mainly concerned with a characterization of integrable metrics in the sense of Question \eqref{itm:Q2} for the two-dimensional torus $\T^2$. In this case, the largest known class of such metrics $g$ are so called \textit{Liouville metrics}, where the line element takes the form
\begin{equation} \label{eq:Liouvillemetric1}
\D s^2 = (f_1(x^1)+ f_2(x^2)) \left((\D x^1)^2 + (\D x^2)^2 \right)\,,
\end{equation}
in appropriate global coordinates $(x^1, x^2)$ and where $f_1$ and $f_2$ are sufficiently regular positive periodic functions. See Section \ref{subsec:linandquad} for more details.

The most important result about integrable Hamiltonian systems is the following well known theorem, establishing the existence of so-called \emph{action-angle coordinates}, which shall be employed in our proofs in Section \ref{sec:proofs}.  
\begin{thm} \label{thm:liouvillearnold} {\rm (Liouville-Arnold Theorem \cite{arnold})} \\
Let $H$ be a Liouville integrable Hamiltonian on $T^*M$ and let 
\begin{equation*}
T_{\boldsymbol{f}} = \left\{  (x,p) \in T^*M : F_i(x) = f_i, \, i= 1,...,n \right\}
\end{equation*}
be a regular level surface of the first integrals $F_1, ... , F_n$. Then we have the following: 
\begin{itemize}
\item[(a)] The level set $T_{\boldsymbol{f}} \subset T^*M$ is a smooth submanifold of dimension $n$ that is invariant under the geodesic flow. Any compact connected component of $T_{\boldsymbol{f}}$ (again denoted by $T_{\boldsymbol{f}}$) is diffeomorphic to an $n$-dimensional torus $\T^n$, called a {\rm Liouville torus}.
\item[(b)] There exists a neighborhood $U$ of $T_{\boldsymbol{f}}$ and a coordinate system $ (\theta, I) : U \to \T^n \times \R^n$ with $\omega = \sum_{i=1}^{n}\D \theta^i \wedge \D I_i$, called {\rm action-angle variables}, such that $ T_{\boldsymbol{f}} = I^{-1}(0)$ is a level set of the action variables and $F_i = F_i(I)$. Therefore, the Hamiltonian equations~\eqref{eq:Hamiltonianequations} take the form 
\begin{equation} \label{eq:actionanglecoordinates}
\begin{cases}
\dot{I}_i =& 0 \\
\dot{\theta}^i =& \omega_i(I_1, ... , I_n)\,.
\end{cases}
\end{equation}
\end{itemize} 
\end{thm}
\subsection{Maupertuis principle} \label{subsec:maupertuis}
In order to approach the questions \eqref{itm:Q1} and \eqref{itm:Q2}, we will utilize the \textit{Maupertuis Principle} (see, e.g., \cite{bolsfomkoz1995}): For a compact Riemannian manifold, $(M,g)$, let 
\begin{equation} \label{eq:Hamiltonian2}
H(x,p) = \frac{1}{2}\sum_{ij} g^{ij}(x)p_ip_j - V(x)
\end{equation}
be a natural mechanical Hamiltonian function on $T^*M$, where $V \in C^2(M)$ denotes some potential function. Moreover, let $T_h = \{ H(x,p) = h\}$ be an isoenergy submanifold for some $h > - \min_x V(x)$ and note that $T_h$ is also an isoenergy submanifold for another system with  Hamiltonian function 
\begin{equation*}
\widetilde{H}(x,p) = \frac{1}{2}\sum_{ij} \frac{g^{ij}(x)}{h+V(x)}p_ip_j\,,
\end{equation*}
i.e.~$T_h = \{\widetilde{H}(x,p) = 1\}$. Now, the \emph{Maupertuis principle} states that the integral curves for the Hamiltonian vector fields $X_H$ and $X_{\widetilde{H}}$ on the fixed isoenergy submanifold $T_h$ coincide. Moreover, if there exists an additional first integral $F$ for $H$ on $T_h$, then there also exists a first integral $\tilde{F}$ for $\widetilde{H}$ on the \emph{whole} of $T^*M$ (except, potentially, at the zero section). 
Finally, note that the vector field $X_{\widetilde{H}}$ gives rise to the geodesic flow of the Riemannian metric $\widetilde{g}$ with 
\begin{equation} \label{eq:transformedmetric}
\widetilde{g}_{ij}(x) = (h+V(x))g_{ij}(x)\,,
\end{equation}
which is the correspondence between Hamiltonian systems and geodesic flows we will use.

\section{Main results} \label{sec:Mainresults}
The main results of this paper are rigidity results in the sense of Question \eqref{itm:Q2} for classes of integrable metrics on the two-torus $\mathbb{T}^2 = \R^2 /\Gamma$, initially equipped with the flat metric, and hence obtained by a Hamiltonian defined on $T^*\T^2$ by means of the Maupertuis principle. In general, $\Gamma \subset \R^2$ is an arbitrary lattice, but we focus on the case $\Gamma = \Z^2$ here. We define the Hamiltonian function
\begin{equation} \label{eq:H0}
H_0(x,p) = \frac{p_1^2}{2}+\frac{p_2^2}{2}  - \mu_1\, V_1(x^1) - \mu_2 \, V_2(x^2)
\end{equation}
 on $T^*\mathbb{T}^2$, where $\mu_i \in [0,\infty)$ are parameters, and $V_i \in C^2(\mathbb{T})$ with $V_i \ge 0$ and $\Vert V_i \Vert_{C^0} \le \mathcal{C}_i$
are Morse functions (or constant). We may assume w.l.o.g.~that $\min_{x^i} V_i(x^i) = 0$. This includes, e.g., the situation of two pendulums, i.e.~$V_i(x^i) = 1- \cos(2\pi x^i)$. The torus coordinates are denoted by $x = (x^1,x^2) \in \mathbb{T}^2$ and the conjugate coordinate pairs are $(x^1,p_1)$ and $(x^2,p_2)$. By the Maupertuis principle, for fixed $e>0$, the Hamiltonian flow on the isoenergy manifold $T_{e} = \{H_0 = e \}$ coincides with the geodesic flow on $\mathbb{T}^2$ with the Liouville metric $g_{e}$ (see eq.~\eqref{eq:Liouvillemetric1} and Section \ref{subsec:linandquad} for more details) having line element
\begin{equation*}
\D s^2_e = \left(e+\mu_1 \, V_1(x^1) +\mu_2 \, V_2(x^2) \right) \left( (\D x^1)^2 + (\D x^2)^2\right)\,.
\end{equation*}
The system with Hamiltonian function \eqref{eq:H0} is clearly integrable in the sense of Definition~\ref{def:integrable}, since an additional conserved quantity can easily be found as
\begin{equation} \label{eq:first int}
F_1(x,p) = \frac{p_1^2}{2} - \mu_1 \, V_1(x^1)\,.
\end{equation}
The Liouville foliation of $T_{e}$ has the following qualitative structure, that is similar to the phase portrait of the pendulum. The common level surface 
 \begin{equation*}
T_{(e,f)} = \{ H_0 = e, \ F_1 = f  \}
 \end{equation*}
 differs in shape, depending on the values of $e$ and $f$. Recall that $e >0$ and $V_i \ge 0$. If (i) $f\in (- \mu_1 \, \max_{x^1} V_1(x^1), 0)$ and $e-f > 0$, $T_{(e,f)}$ is an \textit{annulus};  if (ii) $f > 0$ and $e-f >0$, $T_{(e,f)}$ is a \textit{torus}; if (iii) $f > 0 $ and $e-f \in (- \mu_2\, \max_{x^2} V_2(x^2), 0)$, $T_{(e,f)}$ is an \textit{annulus}. Therefore, if $V_1$ and $V_2$ are both non-constant, the foliation qualitatively exhibits  a pendulum-like phase portrait (see Figure \ref{fig:1}). 
\begin{SCfigure}[1.15][htb]
	\centering
	\includegraphics[width=0.65\textwidth]{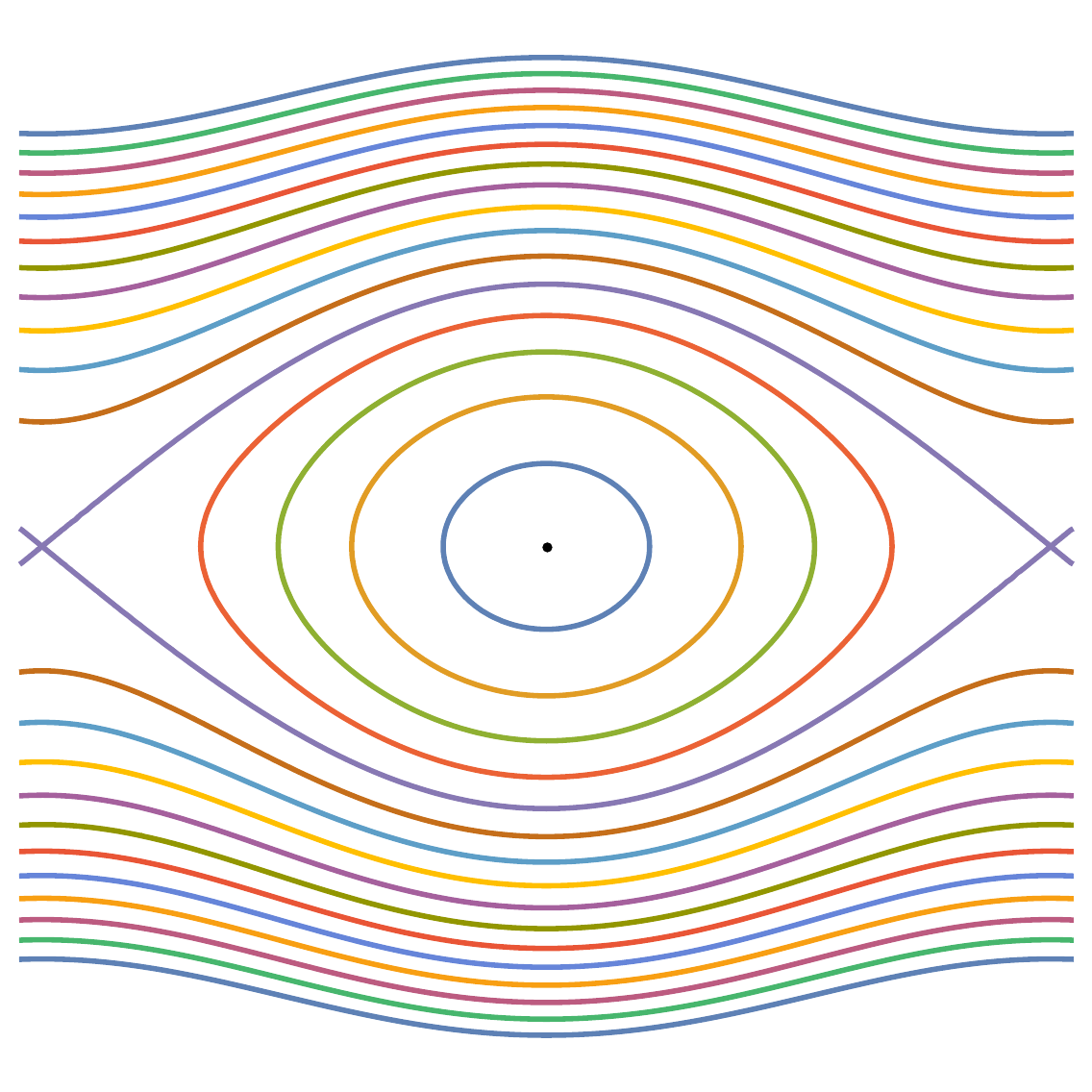}
	
	\caption[]{Schematic picture of the Liouville foliation of the phase space $T^*\T \cong \T \times \R$ for the classical one-dimensional pendulum system described by the Hamiltonian function
		\begin{minipage}{\linewidth}
			\begin{equation*}
				H(x,p) = \frac{p^2}{2} - \big( 1- \cos(2 \pi x) \big)\,.
			\end{equation*}
		\end{minipage}
	\\[2mm]
	The horizontal direction covers slightly more than one period of length one. 
	}
	\label{fig:1}
\end{SCfigure}
\subsection{Definitions and assumptions}
Our main results concern perturbations of the Hamiltonian function \eqref{eq:H0} in the class of mechanical systems as 
\begin{equation} \label{eq:Heps}
H_\epsi(x,p) = H_0(x,p) + \epsi U(x)\,,
\end{equation}
where $\epsi \in \R$ and $U \in C^2(\mathbb{T}^2)$ denotes a perturbing potential, which is assumed to 
be a Morse function (or a constant) and 
have an absolutely convergent Fourier series\footnote{Note that in two dimensions, $C^2$-regularity is not sufficient for ensuring an absolutely convergent Fourier series, although in one dimension it is.}
\begin{equation*} \label{eq:Fourierdecomp}
U(x) = \sum_{k_1 \in \Z}  U_{k_1}(x^2) \, \E^{\I 2 \pi k_1 x^1 }= \sum_{(k_1,k_2)  \in \Z^2} U_{k_1,k_2} \, \E^{\I 2\pi (k_1 x^1 + k_2 x^2)}\,.
\end{equation*}
In the following, we introduce several subsets of $\Z^2$ in such a way, that their definitions immediately carry over in arbitrary dimension $d \in \N$ (see Remark \ref{rmk:higher dim}). 
First, we define the \emph{spectrum} of $U$, i.e.~the set of non-vanishing Fourier coefficients, as
\begin{equation} \label{eq:spectrum}
\mathcal{S}_U := \left\{ \boldsymbol{k} = (k_1, k_2) \in \Z^2 : U_{\boldsymbol{k}} \neq 0  \right\}\,,
\end{equation}
while the \emph{non-singular} spectrum is denoted by
\begin{equation} \label{eq:nonsinspec}
\mathcal{S}_{U,0} := \left\{ \boldsymbol{k} \in \mathcal{S}_U : \exists i \neq j \ \mathrm{s.th.}\  k_i \cdot k_j \neq 0  \right\}\,. 
\end{equation}
Moreover, we define the \emph{coprime set of the orthogonal complement of $\mathcal{S}_U$} as well as its \emph{non-singular subset} via 
\begin{equation} \label{eq:basis0perp}
\mathcal{B}(\mathcal{S}_U^\perp) := \left\{  \boldsymbol{b} \in \mathcal{S}_U^\perp : \boldsymbol{b} \ \mathrm{coprime} \right\} \quad \text{and} \quad \mathcal{B}_0(\mathcal{S}^\perp_U) := \left\{  \boldsymbol{b} \in \mathcal{B}(\mathcal{S}^\perp_U) : \prod_i b_i \neq 0 \right\}\,,
\end{equation}
respectively. Note that the orthogonal complement is taken within $\Z^2$. For the proofs in Section~\ref{sec:proofs} and the generalization in Appendix \ref{app:higher dim} it is important to observe that for every $\boldsymbol{k} \in \mathcal{S}_{U,0}$ exists some $\boldsymbol{b} \in \mathcal{B}_0(\mathcal{S}_U^\perp)$ such that $\boldsymbol{b} \cdot \boldsymbol{k} = 0$.

Our main results will be formulated under the following assumptions. 
\\[5mm]
\textbf{1.~Assumptions on the perturbed Hamiltonian function \eqref{eq:Heps}.} \\ Let $H_0 \in C^2(T^*\T^2)$ denote the Hamiltonian function from \eqref{eq:H0} with $\Vert V_i \Vert_{C^0} \le \mathcal{C}_i$ and $\mu_i \in [0,\tilde{\mu}_i]$ for some $\tilde{\mu_i} \in [0,\infty)$, $i \in \set{1,2}$, and $U$ be a perturbing potential as in \eqref{eq:Heps}, which satisfies one of the following assumptions. 
\begin{itemize}
	\item[\bf(\mylabel{itm:A1}{A1})] \label{itm:A1}If $\tilde{\mu}_1 = \tilde{\mu}_2 = 0$, we have $U \in C^2(\mathbb{T}^2)$. 
\item[\bf(\mylabel{itm:A2}{A2})] \label{itm:A2} If, w.l.o.g., $\tilde{\mu}_1 = 0$ and $\tilde{\mu}_2 > 0$, we have $U \in C^2(\mathbb{T}^2)$ 
	and there exists $d^{(2)} \ge 0$ such that 
		\begin{equation} \label{eq:A2}
	\mathcal{S}_U \subset \Z \times \big[-d^{(2)}, d^{(2)}\big] \,,
	\end{equation}
	i.e.~$U$ is a trigonometric polynomial in the second variable $x^2$. 
	\item[\bf(\mylabel{itm:A3}{A3})] \label{itm:A3} If $\tilde{\mu}_1, \tilde{\mu}_2 > 0$, we have $U \in C^2(\mathbb{T}^2)$ 
		and there exist $d^{(1)}, d^{(2)}\ge 0$ such that 
	\begin{equation} \label{eq:A3}
	\mathcal{S}_U \subset \big[- d^{(1)}, d^{(1)}\big] \times \big[- d^{(2)}, d^{(2)}\big]\,,
	\end{equation}
	i.e.~$U$ is a trigonometric polynomial. 
\end{itemize}
We denote the minimum over all $d^{(i)}$ such that \eqref{eq:A2} resp.~\eqref{eq:A3} holds as $\deg_U^{(i)}$ and call it the \textit{$i$-degree of $U$}. Whenever we refer to one of the Assumptions \eqref{itm:A1}, \eqref{itm:A2}, or \eqref{itm:A3}, we implicitly assume that $H_0 \in C^2(T^*\T^2)$ is of the form \eqref{eq:H0}.  
\\[5mm]
Note that the assumption on the spectrum \eqref{eq:spectrum} of $U$ is more restrictive when we include more general potentials $\mu_1 V_1$ and $\mu_2 V_2$ in the unperturbed Hamiltonian $H_0$ \eqref{eq:H0}. 

The following basic proposition is fundamental for the precise formulation of our assumptions concerning preservation of integrability. It rephrases certain aspects of the Liouville-Arnold Theorem~\ref{thm:liouvillearnold} in our concrete setting using standard notions from weak KAM theory (see Appendix~\ref{app:weakKAM}). 
\begin{prop} {\rm (Liouville-Arnold Theorem and weak KAM theory \cite{SorrentinoLecNotes})}\label{prop:KAM} \\
		Let $H_0 \in C^2(T^*\T^2)$ be the Hamiltonian function from \eqref{eq:H0}.
	\begin{itemize}
\item[(a)] In the region of phase space, where $f > 0$ as well as $e - f > 0$, each of the two connected component of a Liouville torus $T_{(e, f)}$ (again denoted by $T_{(e, f)}$) is a Lipschitz\footnote{We will see in Appendix \ref{app:weakKAM} that $u_\cC \in C^3(\T^2)$, so the regularity of $T_{(e,f)}$ is in fact $C^2$.} Lagrangian graph, i.e.
\begin{equation*}
T_{(e, f)} = \left\{ (x, \cC + \nabla_x u_{\cC}) : x \in \T^2 \right\}
\end{equation*}
for a unique cohomology class $\cC \in H^1(\T^2, \R) \cong \R^2$ with $|c_i| > \sqrt{\mu_i} \, \mathfrak{c}(V_i)$ and $u_\cC \in C^{1,1}(\T^2)$,\footnote{Here, $\mathfrak{c}(V_i):= \int_{0}^{1} \sqrt{2 \,  V_i(x^i)} \, \D x^i$ (see Appendix \ref{app:weakKAM}) and $C^{1,1}$ denotes the functions in $C^1$ with Lipschitz derivative.}  so we may equivalently write $T_{(e, f)} \equiv T_\cC$.
The function $u_\cC \in C^{1,1}(\T^2)$ is a classical solution of the 
Hamilton-Jacobi equation 
\begin{equation*}
\alpha_{}(\cC) = H_0(x, \cC + \nabla_x u_{ \cC}(x))\,,
\end{equation*}
where the lhs.~is Mather's $\alpha$-function (see Appendix \ref{app:weakKAM}). 
\item[(b)] The Hamiltonian flow on $T_\cC$ is conjugated to a rotation on $\T^2$, i.e.~there exists a diffeomorphism $\varphi: \T^2 \to T_\cC$ such that $\varphi^{-1}\circ \Phi_t^{X_H} \circ \varphi = R_t^{\boldsymbol{\omega}}$, $\forall t \in \R$,  where $R_t^{\boldsymbol{\omega}}: \T^2 \to \T^2, \, x \mapsto (x + \boldsymbol{\omega} t \mod \Z^2)$ for some rotation vector $\boldsymbol{\omega} \in \R^2$. 
	\end{itemize}
\end{prop}
An invariant Liouville torus $T_\cC$ is called \emph{irrational} or \emph{non-resonant}, if ${\K} \cdot \boldsymbol{\omega}  \neq 0$ for all $\boldsymbol{k} \in \Z^2 \setminus \{ 0 \}$. If this is not the case, the invariant torus is \emph{rational} or \emph{resonant}. We shall also call it an \emph{(invariant) $\boldsymbol{k}$-torus}, in case that $\boldsymbol{k} \cdot \omf = 0$. For two-dimensional manifolds (and if $\omega_2 \neq 0$), this can be phrased as a distinction between $\omega_1/\omega_2 \notin \mathbb{Q}$ and $\omega_1/\omega_2 \in \mathbb{Q}$. 
\\[5mm]
\noindent\textbf{2.~Assumptions on the preserved integrability of \eqref{eq:Heps}.} \\*
Let $H_0 \in C^2(T^*\T^2)$ denote the Hamiltonian function from \eqref{eq:H0} satisfying one of the Assumptions \eqref{itm:A1} - \eqref{itm:A3}, and $U$ a perturbing potential as in \eqref{eq:Heps} such that the following statement concerning the perturbed Hamilton-Jacobi equation (HJE)
\begin{equation} \label{eq:HJE}
\alpha_{\epsi}(\cC) = H_\epsi(x, \cC + \nabla_x u_{\epsi, \cC}(x) )
\end{equation}
as well as the preserved integrability of $H_\epsi$ is satisfied. 

\begin{itemize} 
\item[\bf(\mylabel{itm:P}{P})] \label{itm:P}There exists an energy $e >0$, such that for every $(n,m) \in \mathcal{B}_0(\mathcal{S}_U^\perp)$ (recall \eqref{eq:basis0perp}) and $\mu_i \in [0,\tilde{\mu}_i]$, $i \in \set{1,2}$, there exists a sequence $(\epsi_k)_{k \in \N}$ with $\epsi_k \neq 0$ but $\epsi_k \to 0$ such that we have the following:  
	\begin{itemize}
		\item[(i)] The $(n,m)$-torus from Proposition~\ref{prop:KAM}  characterized by $\boldsymbol{c} \in H^1(\T^2, \R)$ with 
		\begin{equation} \label{eq:ccond}
		|c_i|> \sqrt{\mu_i}\, \mathfrak{c}(V_i)
		\end{equation}
		in the isoenergy submanifold $T_e$ is preserved under the sequence of deformations $(H_{\epsi_k})_{k \in \N}$. 
		\item[(ii)] For $\cC \in H^1(\T^2, \R)$ satisfying \eqref{eq:ccond}, Mather's $\alpha$-function and a solution $u_{\epsi,\cC}:\T^2 \to \mathbb{R}$ of the HJE \eqref{eq:HJE} can be expanded to first order in $\epsi$, i.e. 
		\begin{equation} \label{eq:propertyP1}
		u_{\epsi, \cC} = u_{\cC}^{(0)} + \epsi u_{ \cC}^{(1)} + \mathcal{O}_\cC(\epsi^2)\quad \text{and} \quad \alpha_{\epsi} = \alpha^{(0)} + \epsi \alpha^{(1)} + \mathcal{O}(\epsi^2)\,,
		\end{equation}  
		where $u_{\cC}^{(0)} ,\, u_{ \cC}^{(1)} \in C^{1,1}(\T^2)$  and $O_\cC(\epsi^2)$ is understood in $C^{1,1}$-sense.\footnote{Having $C^1$-regularity here would be sufficient for our proofs in Section \ref{sec:proofs}. However, we chose $C^{1,1}$-regularity for the formulation of Assumption \eqref{itm:P} to be in agreement with the statement from Proposition~\ref{prop:KAM}~(b). More precisely, $C^{1,1}$-regularity is kind of a compromise between the true $C^3$-regularity of $u_\cC$ and the required $C^1$-regularity of $u_{\epsi, \cC}$. In addition, $C^{1,1}$ is the optimal regularity for \emph{sub}solutions of \eqref{eq:HJE}, which exist, even if the Hamiltonian $H_\epsi$ is \emph{not} integrable (see \cite{fathisic, bernard}).}
	\end{itemize}
\end{itemize}
 We comment on the validity of assuming \eqref{itm:P} in Remark \ref{rmk:propertyP1} in Appendix \ref{app:weakKAM}. Moreover, we shall also discuss an alternative to \eqref{eq:propertyP1} in Remark \ref{rmk:alterP(iii)}. Finally, one can easily see from the proofs given in Section \ref{sec:proofs}, that the condition on a \emph{fixed} isoenergy manifold $\{ H_\epsi = e\}$ can be relaxed to having preservation of invariant tori in isoenergy manifolds characterized by energies $e \ge e_0$ for some fixed $e_0 > 0$. 
\\[5mm]
Note that the rational invariant tori are the most `fragile' objects of an integrable system as the KAM Theorem \cite{KAM_K, KAM_A, KAM_M} predicts that general (non-integrable) perturbations preserve only `sufficiently irrational' (\emph{Diophantine}) invariant tori.
\subsection{Results}
As mentioned above, our main results in Theorem~\ref{thm:1}, Theorem~\ref{thm:2}, and Theorem~\ref{thm:3} concern rigidity of certain deformations of integrable metrics (in the sense of Question~\eqref{itm:Q2}), which, by means of the Maupertuis principle, correspond to perturbations of the form \eqref{eq:Heps}. More precisely, under the assumptions formulated above, our results show that the perturbed Hamiltonian function \eqref{eq:Heps} has to be of the same general form as the unperturbed Hamiltonian function \eqref{eq:H0}. This means, that the potential $U$ is \emph{separable}, i.e.~there exist $U_1,U_2 \in C^2(\T^2)$ such that 
\begin{equation*}
U(x) = U_1(x^1) + U_2(x^2)\,.
\end{equation*}

\begin{theorem} \label{thm:1} Let $H_\epsi$ from \eqref{eq:Heps} satisfy Assumption \eqref{itm:A1} and Assumption \eqref{itm:P} for some energy $e > 0$. Then $U$ is separable in a sum of two single-valued functions. 
\end{theorem}
Put briefly, in view of of the Maupertuis principle, this means that integrable deformations in the same conformal class of a flat metric are Liouville metrics. 
Now, Theorem~\ref{thm:2} generalizes Theorem \ref{thm:1} to Hamiltonian functions which depend on one toral position variable via a mechanical potential. 
\begin{theorem} \label{thm:2}
Let $H_\epsi$ from \eqref{eq:Heps} satisfy Assumption \eqref{itm:A2}
and Assumption~\eqref{itm:P} for some energy $e > 0$. 
Then the following holds: 
\begin{itemize}
	\item[(a)] If $\tilde{\mu}_2 = \tilde{\mu}_2(\mathcal{C}_2, \deg_U^{(2)}, e)>0$ is small enough (see Lemma~\ref{lem:continuity}), we have that $U$ is separable in a sum of two single-valued functions. 
	\item[(b)] If, additionally, $V_2$ is analytic, then $U$ is separable, irrespective of $\tilde{\mu}_2 > 0$, but only for $\mu_2 \in [0,\tilde{\mu}_2]$ outside of an exceptional null-set. 
\end{itemize}

\end{theorem}
Therefore, by means of the Maupertuis principle, we infer that integrable deformations in the same conformal class of metrics realizing surfaces of revolution (see Section \ref{subsec:linandquad}) are Liouville metrics.
Finally, Theorem \ref{thm:3} generalizes the above results to Hamiltonian functions, which correspond to arbitrary Liouville metrics by means of the Maupertuis principle. 

 \begin{theorem} \label{thm:3}
Let $H_\epsi$ from \eqref{eq:Heps} satisfy Assumption \eqref{itm:A3} and Assumption \eqref{itm:P} for some energy $e > 0$. Then the following holds: 
\begin{itemize}
	\item[(a)] If $\tilde{\mu}_1 = \tilde{\mu}_1(\mathcal{C}_1, \deg_U^{(1)}, \deg_U^{(2)}, e) >0$ and $\tilde{\mu}_2 = \tilde{\mu}_2(\mathcal{C}_2,  \deg_U^{(1)}, \deg_U^{(2)}, e)>0$ are small enough (see Lemma~\ref{lem:continuity2}), we have that $U$ is separable in a sum of two single-valued functions.
	\item[(b)] If, additionally, $V_2$ is analytic and $\tilde{\mu}_1 = \tilde{\mu}_1(\mathcal{C}_2, \deg_U^{(1)}, \deg_U^{(2)},e) > 0$ is small enough, then $U$ is separable, irrespective of $\tilde{\mu}_2 > 0$, but only for $\mu_2\in [0,\tilde{\mu}_2]$ outside of an exceptional one-dimensional null-set (depending on $\mu_1 \in [0,\tilde{\mu}_1]$). 
	\item[(c)] If both, $V_i$ for $i = 1, 2$, are analytic, then $U$ is separable, irrespective of $\tilde{\mu}_1, \tilde{\mu}_2 > 0$, but only for $(\mu_1, \mu_2)\in [0,\tilde{\mu}_1] \times [0,\tilde{\mu}_2]$ outside of an exceptional two-dimensional null-set. 
\end{itemize} 
 \end{theorem}

Our results formulated in Theorem \ref{thm:1}, Theorem \ref{thm:2}, and Theorem \ref{thm:3} can each be viewed as a verification of a special case of the following conjecture, saying that \emph{`(nice) integrable deformations of Liouville metrics are Liouville metrics'}. 
\\[3mm]
\textbf{Conjecture: Deformational rigidity of Liouville metrics.} \\*{\it Let $g$ be a Liouville metric on $\T^2$ and let $(g_t)_{t \in [0,1]}$ with $g_0 = g$ be a deformation that preserves all rational invariant tori (except finitely many). Then $g_t$ is a Liouville metric for all $t \in [0,1]$. }
\\[1mm]
This conjecture is in strong analogy to the perturbative Birkhoff conjecture for integrable billiards, which is discussed in Section \ref{subsec:Billiard} below.

\section{Integrable metrics on the two-dimensional torus} \label{sec:integrablemetricsonT2}
As pointed out in Section \ref{subsec:intriemmet}, integrability of metrics on one-dimensional manifolds is not questionable and the first non-trivial examples occur whenever $M$ has dimension two. Recall from Definition \ref{def:integrable} that integrability of metrics on two-dimensional manifolds requires the existence of only one additional first integral (beside the Hamiltonian). 
\subsection{Topological obstructions}
The following Theorem due to Kozlov \cite{kozlov1979, kozlov1988} categorizes two-dimensional compact manifolds regarding the possibility to endow them with an integrable metric (see Question~\eqref{itm:Q1}). 
\begin{thm} {\rm (Kozlov \cite{kozlov1979, kozlov1988})}\label{thm:kozlovtopology} \\
Let $M$ be a two-dimensional compact and real-analytic manifold that is endowed with a real-analytic Riemannian metric $g$. If the Euler characteristic $\chi_M$ of $M$ is negative, then there exists no other non-trivial real-analytic first integral. 
\end{thm}
A result similar to Theorem \ref{thm:kozlovtopology} holds for polynomially integrable geodesic flows. 
\begin{thm}{\rm (Kolokoltsov \cite{kolokoltsov})} \\
There exist no polynomially integrable geodesic flow on a closed two-dimensional Riemannian manifold $M$ with negative Euler characteristic $\chi_M$. 
\end{thm}
Recall that any two-dimensional compact manifold $M$ can be represented either as the sphere with handles or the sphere with Möbius strips, in the orientable and non-orientable case, respectively. The Euler characteristic $\chi_M$ can be computed as
\begin{equation*}
\chi_M = 2-2g \qquad \text{resp.} \qquad \chi_M = 2-m\,,
\end{equation*}
where $g$ is the number of handles (the genus) and $m$ is the number of Möbius strips. In order to have integrability, the above theorem imposes the condition $\chi_M\ge 0$ on $M$ and we thus know that the number of handles is at most $1$ and the number of Möbius strips is not greater than $2$. Therefore, any real-analytic two-dimensional compact Riemannian manifold $(M,g)$ with real-analytic (or polynomial) additional integral is either the sphere $\Sph^2$ or the torus $\T^2$ (in the orientable case), or the projective plane $P\R^2$ or the Klein bottle $\mathbb{K}^2$ (in the non-orientable case).\footnote{In \cite{BolsTaim}, Bolsinov and Taimanov give a striking example of a real-analytic Riemannian manifold of dimension three, whose geodesic flow has the peculiar property, that it is smoothly (but not analytically) integrable although it has positive topological entropy \cite{topentropy}. The problem of proving (non-)existence of smoothly (but not analytically) integrable geodesic flows on compact surfaces of genus $g > 1$ is widely open (see \cite{probl1}). }

In this work, we focus on integrable metrics on the torus $\T^2$ and refer to works by Bolsinov, Fomenko, Matveev, Kolokoltsov and others \cite{bolsmatfom1998, fomenkomatveev1988, kolokoltsov, Nguyenpolyaselia1993} for studies on integrable metrics on the sphere, the projective plane, and the Klein bottle. See \cite{probl1, probl2} for recent surveys on open problems and questions concerning geodesics and integrability of finite-dimensional systems in general.
\subsection{Linearly and quadratically integrable metrics} \label{subsec:linandquad}
The first non-trivial class of integrable metrics on the torus $\T^2$ are {\it surfaces of revolution}. Consider a two-dimensional surface $M \subset \R^3$ given by the equation $r = r(z)$ in standard cylindrical coordinates $(r, \varphi, z) \in (0,\infty) \times [0,2\pi) \times \R$. As local coordinates on $M$ we take $z$ and $\varphi$. In case that $r(z)$ is $L$-periodic and we identify $0$ and $L$, then $M$ is diffeomorphic to the torus $\T^2$ and the Riemannian metric induced on $M$ by the Euclidean metric on $\R^3$ has line element
\begin{equation} \label{eq:surfaceofrevolution}
\D s^2 = (1+r'(z)^2) \D z^2 + r(z)^2 \D \varphi^2\,.
\end{equation}
Since the corresponding Hamiltonian function \eqref{eq:Hamiltonian1} is independent of $\varphi$, its associated momentum variable $p_\varphi$ is an additional first integral and thus the metric \eqref{eq:surfaceofrevolution} is integrable. Note that the additional first integral is linear in the momentum variables. 

As discussed earlier, a Riemannian metric $g$ on $\T^2$ is called a \textit{Liouville metric}, whenever its line element can be written as 
\begin{equation*} \label{eq:Liouvillemetric}
\D s^2 = \big(f_1(x^1)+ f_2(x^2)\big) \big((\D x^1)^2 + (\D x^2)^2\big)
\end{equation*}
in appropriate global coordinates $(x^1, x^2)$ and where $f_1$ and $f_2$ are smooth positive periodic functions. The corresponding Hamiltonian function \eqref{eq:Hamiltonian1} is given by 
\begin{equation*}
H(x^1,x^2,p_1,p_2) = \frac{p_1^2 + p_1^2}{2\, (f_1(x^1)+f_2(x^2))}
\end{equation*}
and an additional first integral can easily be obtained as
\begin{equation*}
F(x^1,x^2,p_1,p_2) = p_1^2 - f_1(x^1)\, H(x^1,x^2,p_1,p_2)\,.
\end{equation*}
Therefore, clearly, also Liouville metrics are integrable. Note that the additional first integral $F$ is quadratic in the momentum variables. 
It is not hard to see that a surface of revolution is just a particular case of a Liouville metric, where one can choose, e.g., $f_2 \equiv 0$, by employing a simple change of variables. 

The following proposition also provides the converse to the observation that surfaces of revolution and Liouville metrics admit additional first integrals which are linear and quadratic in the momenta, respectively. It collects several statements that have been proven in early works by Dini \cite{dini}, Darboux \cite{darboux}, and Birkhoff \cite{birkhoff1927}, and were further developed by Babenko and Nekhoroshev \cite{babenkoneko1995}, Kiyohara \cite{kiyohara1991}, Kolokoltsov \cite{kolokoltsov}, and others.
\begin{prop} {\rm (Linear and quadratic first integrals \cite{dini, darboux, birkhoff1927,babenkoneko1995,kiyohara1991,kolokoltsov})} \label{prop:linearandquadratic} 
\begin{itemize}
\item[(a)] Let the metric $g$ on $\mathbb{T}^2$ possess an additional first integral $F$ that is linear in the momenta. Then there exist global periodic coordinates $(x^1,x^2)$ on the torus such that the line element of $g$ takes the form
\begin{equation*}
\D s^2 = f(x^1)\, \big(a\, (\D x^1)^2 + c \, \D x^1\, \D x^2 + b\,(\D x^2)^2\big)\,,
\end{equation*} 
where $f$ is some positive periodic function and $a,b,c \in \R$ such that the quadratic form $aa\, (\D x^1)^2 + c \, \D x^1\, \D x^2 + b\,(\D x^2)^2$ is positive definite. 

Conversely, any such metric on the torus $\T^2$ admits an additional first integral that is linear in the momentum variables. 

In case a linear in momenta $F$ exists locally near a point $q \in \T^2$, then there exists local coordinates $(x^1,x^2)$ near $q$ such that the line element of $g$ reads
\begin{equation*}
\D s^2 = f(x^1) \, \big((\D x^1 )^2+(\D x^2)^2\big)\,.
\end{equation*} 
\item[(b)] A metric $g$ on $\T^2$ possess an additional first integral $F$ that is quadratic in the momenta if and only if there exists a finite-sheeted covering $\pi :\widetilde{\T}^2\to \T^2$ by another torus, such that the lifted metric $\widetilde{g} = \pi^* g$ is globally Liouville, i.e.~there exist global periodic coordinates $(x^1,x^2)$ on $\widetilde{\T}^2$ and smooth positive periodic functions $f_1$ and $f_2$ such that the line element of $\widetilde{g}$ takes the form
\begin{equation*}
\D \widetilde{s}^2 = \big(f_1(x^1)+ f_2(x^2)\big) \,  \big((\D x^1)^2 + (\D x^2)^2\big)\,.
\end{equation*}
There exist Riemannian metrics $g$ on $\T^2$ which are not globally Liouville but have an additional first integral that is quadratic in the momentum variables. 

In case a quadratic in momenta $F$ exists locally near a point $q \in \T^2$, then there exist local coordinates $(x^1,x^2)$ near $q$ such that the line element of $g$ reads
\begin{equation*}
\D {s}^2 = \big(f_1(x^1)+ f_2(x^2)\big) \, \big((\D x^1)^2 + (\D x^2)^2\big)\,.
\end{equation*}
\end{itemize}
\end{prop}
This classical result completely characterizes the integrable metrics $g$ on $\T^2$ that admit an additional first integral that is linear or quadratic in the momentum variables. Similar results hold for Riemannian metrics on general two-dimensional manifolds~\cite{bolsmatfom1998, fomenkomatveev1988, kolokoltsov, Nguyenpolyaselia1993}. 
\subsection{Polynomially integrable metrics of higher degree}
In the case of a sphere $\Sph^2$, one can easily construct examples of metrics which admit an additional first integral that is cubic resp.~quartic in the momentum variables. Using the Maupertuis principle, these can be obtained from the metrics constructed from Goryachev-Chaplygin \cite{goryachev,chaplygin} and Kovaleskaya \cite{kovaleskaya} in the situation of the dynamics of a rigid body. Therefore, let $h>1$ be large enough (cf.~\eqref{eq:transformedmetric}) and define the metrics $g_3$ and $g_4$ on $\R^3$ via their respective line elements
\begin{equation*}
 \D s^2_3 = \frac{h-x^1}{4} \, \frac{(\D x^1)^2 + (\D x^2)^2 + 4 (\D x^3)^2}{(x^1)^2 + (x^2)^2 + (x^3)^2/4}\,, \quad   \D s^2_4 = \frac{h-x^1}{2} \, \frac{(\D x^1)^2 + (\D x^2)^2 + 2 (\D x^3)^2}{(x^1)^2 + (x^2)^2 + (x^3)^2/2}\,.
\end{equation*}
By restriction of $g_3$ and $g_4$ to the unit sphere $\Sph^2 \subset \R^3$, the resulting metrics admit an additional first integral that is cubic resp.~quartic in the momentum variables. It was shown by Bolsinov, Fomenko and Kozlov \cite{bolsinovfomenko1994, bolsfomkoz1995} that these cannot be reduced to first integrals that are polynomially in the momentum variables of a lower degree, i.e.~they are not linearly or quadratically integrable. Since all attempts to construct such examples for the case of the torus have failed so far, the following folklore conjecture emerged.
\\[3mm]
\textbf{Folklore Conjecture.} {\it Liouville metrics are the only integrable metrics on $\T^2$.}
\\[3mm]
In this general form, there is strong indication for conjecture being false, as to be shown below (see Theorem \ref{thm:corsikaloshin}). We will, however, provide existing results, which indicate that a certain weaker version of this conjecture, also formulated below, is indeed true. 

It was proven by Korn and Lichtenstein \cite{korn, lichtenstein} that on every point on a two-dimensional Riemannian manifold $(M,g)$ there exist \emph{locally} isothermal coordinates, that is, locally, the line element takes the form
\begin{equation} \label{eq:isothermal}
\D s^2 = \lambda(x^1,x^2) \, \big((\D x^1)^2 + (\D x^2)^2\big)\,,
\end{equation}
where $\lambda$ is some smooth positive function. In the case of a torus, it can be shown (by virtue of the uniformization theorem) that there exist \emph{global} isothermal coordinates (not necessarily periodic), so the metric $g$ is conformal equivalent to the Euclidean metric $g_{\mathrm{eucl}}$. In particular, assuming that $(x^1,x^2)$ are just the angular coordinates on the torus $\T^2$ and in the special case of $\lambda$ being a trigonometric polynomial,\footnote{This means that the spectrum $\mathcal{S}_\lambda$ defined in \eqref{eq:spectrum} is bounded. } we have the following result due to Denisova and Kozlov. 
\begin{thm}{\rm (Denisova-Kozlov \cite{denisovakozlov1})} \label{thm:dk1} \\
Let $\lambda$ from \eqref{eq:isothermal} be a trigonometric polynomial and assume that the geodesic flow on $\T^2$ is polynomially integrable. Then there exists an additional polynomial first integral of degree at most two. 
\end{thm}
Note that by Weierstrass's Theorem, any conformal factor $\lambda$ can be approximated as closely as required by a trigonometric polynomial. However, in the case of a general conformal factor $\lambda$, there is the following Theorem, again due to Denisova and Kozlov~\cite{denisovakozlov2}. 
\begin{thm} {\rm (Denisova-Kozlov \cite{denisovakozlov2})} \label{thm:dk2} \\
Assume that the geodesic flow on $(\T^2, g)$ is polynomially integrable with first integral $F$ of degree $n$ such that 
	\begin{itemize}
		\item[(a)] if $n$ is even, then $F$ is an even function of $p_1$ and $p_2$, 
		\item[(b)] if $n$ is odd, then $F$ is an even function of $p_1$ (or $p_2$) and an odd function of $p_2$ (or $p_1$). 
	\end{itemize}
Then there exists an polynomial first integral of degree at most two. 
\end{thm}
In the following Theorem we collect several results from Bialy \cite{bialy1987}, Denisova, Kozlov~\cite{denisovakozlov3}, and Treshchev \cite{denisovakozlovtreshev2012}, Agapov and Aleksandrov \cite{agapov}, and Mironov \cite{mironov}.
\begin{thm} \label{thm:polynomial}
Let $H$ be a natural mechanical Hamiltonian (see \eqref{eq:Hamiltonian2}) on  the torus $\T^2$ equipped with the flat metric $g_{\mathrm{eucl}}$. Assume that $H$ is polynomially integrable of degree $n$. If $n = 3,4$, there exists another polynomial first integral of degree at most two. Whenever $H$ is a real-analytic Hamiltonian, this is also true for $n=5$. 
\end{thm}
Kozlov and Treshchev \cite{kozlovtreshchev} considered the problem from yet another point of view. They investigated the case of a mechanical Hamiltonian 
\begin{equation*}
H = \frac{1}{2} \sum_{ij} a_{ij} p_i p_j + V(x^1, ... , x^n)\,,
\end{equation*}
where $A = (a_{ij})_{ij}$ is a positive definite matrix and $V$ is a trigonometric polynomial of $(x^1, ... , x^n) \in \T^n$. On the one hand, they show that there exist $n$ polynomial first integrals if and only if the spectrum $\mathcal{S}_V$ of $V$ is contained in $m \le n$ mutually orthogonal lines meeting at the origin. On the other hand, they showed that whenever there exist $n$ polynomial integrals with independent forms of highest degree, then there exist $n$ independent involutive polynomial first integrals of degree at most two. {In case that $a_{ij} = \delta_{i,j}$ (which can be achieved by diagonalization and scaling), Combot \cite{combot} improved the first result from the assumption of \emph{polynomial} integrability to \emph{rational} integrability, i.e.~the additional first integrals being rational functions of $p_i$ and $\E^{\I 2 \pi x^i}$.} More recently \cite{sharaf1, HMS, sharaf2}, the problem was rephrased in the language of Killing tensor fields on $\T^2$, where the \emph{order} of an additional (polynomial) first integral is replaced by the \emph{rank} of a Killing tensor filed. 

The results of Theorems \ref{thm:dk1}, \ref{thm:dk2} and \ref{thm:polynomial} support the validity of the following weaker version of the folklore conjecture formulated by Denisova and Kozlov~\cite{denisovakozlov1}.
\\[3mm]
\textbf{Conjecture.~\cite{denisovakozlov1}} \textit{If $g$ is a metric on $\T^2$ that is polynomially integrable, then there exists an additional polynomial first integral of degree at most two. }
\\[3mm]
By Proposition \ref{prop:linearandquadratic} this means that polynomially integrable metrics on $\T^2$ are Liouville metrics. However, beside the partial results given above, a proof of this conjecture is still open. The numerous attempts on proving it used methods of complex analysis \cite{birkhoff1927, babenkoneko1995} and the theory of PDEs \cite{bialymironov1, bialymironov2}. More precisely, it is shown by Kolokoltsov \cite{kolokoltsov} that there exists an additional first integral quadratic in the momenta if and only if there exists a holomorphic function $R(z) = R_1(z) +\I R_2(z)$, with real valued $R_1$ and $R_2$ and $z = x^1+\I x^2$, which solves 
\begin{equation} \label{eq:kolokoltsovpde}
R_2 (\partial_{x^2}^2 \lambda - \partial_{x^1}^2 \lambda) +  R_1 (\partial_{x^1} \partial_{x^2} \lambda) - 3 (\partial_{x^1} R_2)( \partial_{x^1} \lambda) +3 (\partial_{x^2} R_2) (\partial_{x^2} \lambda) + 2 (\partial_{x^2}^2 R_2) \lambda = 0\,,
\end{equation}
where $\lambda$ denotes the conformal factor from \eqref{eq:isothermal}. Note that the second term in \eqref{eq:kolokoltsovpde} disappears whenever $\lambda$ is the conformal factor of a Liouville metric. In this situation, the linear PDE \eqref{eq:kolokoltsovpde} always has a holomorphic solution $R = R_1+\I R_2$. The existence of first integrals of higher degree turns out to be equivalent to delicate questions about non-linear PDEs of hydrodynamic type~\cite{bialymironov1, bialymironov3, bialymironov2}. The PDE-approach has also successfully been applied to generate new examples of integrable \emph{magnetic} geodesic flows as analytic deformations of Liouville metrics on $\T^2$ without magnetic field (see \cite{ABM2017}).

Regarding the original folklore conjecture stated above, there is a result due to Corsi and Kaloshin \cite{corsikaloshin}, which shows it being false in the following (weaker) sense. 
\begin{thm}{\rm (Corsi-Kaloshin \cite{corsikaloshin})} \label{thm:corsikaloshin}\\
There exists a real-analytic mechanical Hamiltonian
\begin{equation*}
H_\epsi(x^1,x^2,p_1,p_2) = \frac{p_1^2 + p_2^2}{2} + U(x^1,x^2;\epsi) 
\end{equation*}
with a non-separable\footnote{The function $U$ is called non-separable whenever it cannot be written as a sum of two single-valued functions.} potential $U$ and an analytic change of variables $\Phi$ such that $H_\epsi \circ \Phi =  (p_1^2 + p_2^2)/2$ on the energy surface $\{ H_\epsi = 1/2 \}$ and $p \in \mathcal{P}$, where $\mathcal{P}$ denotes a certain cone in the action space. 
\end{thm}
If one assumes that the \emph{whole} phase space $T^*\T^2$ is foliated by  two-dimensional invariant Liouville tori (which is often called $C^0$-integrability or \emph{complete integrability}), then it follows from Hopf conjecture \cite{hopf} that the associated metric must be flat.\footnote{Similar results have been shown for geodesic flows of more general Finsler metrics on $\T^2$ preserving a sufficiently regular foliation of the phase space \cite{finsler1, finsler2}}  This notion of integrability is thus too strong for a meaningful characterization of integrable metrics on~$\T^2$. 
\subsection{Analogy to integrable billiards} \label{subsec:Billiard}
 The fundamental question \eqref{itm:Q2} of characterizing integrable metrics on the torus $\T^2$ can be thought of as an analogue of identifying the class of integrable billiards \cite{KS2}. For billiards, integrability is understood in a similar way as for the geodesic flow (see Definition \ref{def:integrable}). More precisely, integrability is characterized either through the existence of an integral of motion (near the boundary of the billiard table) for the so called billiard ball map, or the existence of a foliation of the phase space (globally, or near the boundary), consisting of invariant curves. The \textit{classical Birkhoff conjecture} \cite{birkhoff19272, poritsky} states that the boundary of a strictly convex integrable billiard table is necessarily an ellipse. This corresponds to the folklore conjecture formulated above. Remarkably, while the Birkhoff conjecture is believed to be true, and there is strong evidence that this indeed the case \cite{bialymironov4, glutsyuk, ADK, KS}, the folklore conjecture in its general form was shown to be false by Theorem \ref{thm:corsikaloshin}. 
 
 However, recall that, if one assumes $C^0$-integrability of a metric on $\T^2$, the metric is actually flat. This corresponds to the following result from Bialy in the case of billiards. 
 \begin{thm}{\rm (Bialy \cite{bialy1993})} \label{thm:bialy} \\
If the phase space of the billiard ball map is completely foliated by continuous invariant curves which are all not null-homotopic, then the boundary of the billiard table is a circle. 
 \end{thm} 
Following a similar strategy leading to Theorem \ref{thm:bialy}, Bialy and Mironov \cite{bialymironov5} proved the Birkhoff conjecture for centrally symmetric billiards, assuming only \emph{local} $C^0$-integrability, i.e.~the foliation of a suitable open \emph{proper} subset of the phase space.
Beside this, the weakened version of the folklore conjecture (polynomial integrals can be reduced to integrals of degree at most two) corresponds to the so called \textit{algebraic Birkhoff conjecture}, which has recently been proven \cite{bialymironov4, glutsyuk}.

The main results of this paper in Theorem \ref{thm:1}, Theorem \ref{thm:2}, and Theorem \ref{thm:3} prove special cases of our conjecture that integrable deformations of Liouville metrics which preserve all (but finitely many) rational invariant tori are again Liouville metrics. This is related to the following conjecture in the case of billiards. 
\\[3mm]
\textbf{Perturbative Birkhoff conjecture.~\cite{KS2}} \textit{A smooth strictly convex domain that is sufficiently close to an ellipse and whose corresponding billiard ball map is integrable, is necessarily an ellipse.}
 \\[3mm]
A first result in this direction was obtained by Delshams and Ramírez-Ros \cite{ramros}. More recently, Avila, De Simoi, and Kaloshin \cite{ADK} proved the conjecture for domains which are sufficiently close to a circle. The complete proof for domains sufficiently close to an ellipse of any eccentricity is given by Kaloshin and Sorrentino in \cite{KS}. Both works require the preservation of rational caustics\footnote{A curve $\Gamma$ is a caustic for the billiard in the domain $\Omega$ if every time a trajectory is tangent to it, then it remains tangent after every reflection according to the billiard ball map.} which can be thought of as an analogue for the preservation of rational invariant tori as a fundamental assumption of our main results from Section~\ref{sec:Mainresults}. {The result in \cite{ADK} was later extended by Huang, Kaloshin, and Sorrentino \cite{HKS} to the case of \emph{local integrability} close to the boundary and finally improved by Koval~\cite{Koval}. }

Finally, as shown by Vedyushkina and Fomenko \cite{vedyushkinafomenko}, linearly and quadratically integrable geodesic flows on orientable two-dimensional Riemannian manifolds are Liouville equivalent to topological billiards, glued from planar billiards bounded by concentric circles and arcs of confocal quadrics, respectively. 
\section{Proofs} \label{sec:proofs}
In this Section we prove our main result as formulated in Theorem \ref{thm:1}, Theorem~\ref{thm:2}, and Theorem~\ref{thm:3}. All proofs will, in general, follow the same three step strategy. 
\begin{itemize}
\item[\bf(i)] Transform the unperturbed system $H_0$ in action-angle coordinates (cf.~Theorem~\ref{thm:liouvillearnold}, in particular eq.~\eqref{eq:actionanglecoordinates}). 
\item[\bf(ii)] Derive a first-order harmonic equation for the perturbation by Assumption~\eqref{itm:P}.
\item[\bf(iii)] Annihilate sufficiently many Fourier coefficients of the perturbing potential by proving a certain \emph{full-rank} condition for a naturally associated linear system for each of the three theorems separately (cf.~Lemmas~\ref{lem:annihilation}, \ref{lem:continuity}, and~\ref{lem:continuity2}). Finally, for analytic potentials $V_i$, the extensions of our results beyond the perturbative regime are proven by exploiting the analytic dependence of the linear system  on $\mu_i$ (see Appendix~\ref{app:AApendulum}). 
\end{itemize}

\subsection{Proof of Theorem \ref{thm:1}} \label{subsec:proofthm1}
{\bf \underline{Step (i).}} Fix an energy $e > 0$. Since the Hamiltonian is already in action-angle coordinates (cf.~\eqref{eq:actionanglecoordinates}), we simply change notation and write $(x^i,p_i) = (\theta^i,I_i)$ for $i = 1,2$ as well as $\theta = (\theta^1, \theta^2)$ and $I = (I_1,I_2)$, such that the perturbed Hamiltonian function $H_\epsi$ takes the form 
\begin{equation*}
H_\epsi(\theta, I) = \frac{I_1^2}{2} + \frac{I_2^2}{2} + \epsi U(\theta)\,.
\end{equation*}
\textbf{\underline{Step (ii).}}  By Assumption \eqref{itm:P}, for any $(n,m) \in \mathcal{B}_0(\mathcal{S}_U^\perp)$ (recall \eqref{eq:basis0perp}), we can find (in the isoenergy manifold $T_{e_\epsi}$ with energy $e = e_\epsi$ and $\epsi = \epsi_k $ for some $k \in \N$) a rational invariant  invariant Liouville torus with rotation vector $\omf = (\omega_1, \omega_2)$ which satisfies
\begin{equation} \label{eq:omega1}
\frac{\omega_1}{\omega_2} = \frac{n}{m} \in \mathbb{Q} \quad \text{and} \quad \omf = (c_1,c_2)
\end{equation}
for some $\cC \in H^1(\T^2, \R) \cong \R^2$, which we fix now.

Using Assumption \eqref{itm:P} again, we can expand the Hamilton-Jacobi equation \eqref{eq:HJE} as
\begin{align*}
\alpha_\epsi(\cC) &= H_\epsi(\theta, \cC + \nabla u_{\epsi, \cC}(\theta)) \\
&= \frac{\vert \partial_{\theta^1} u_{\epsi, \cC}(\theta) + c_1\vert^2}{2} + \frac{\vert \partial_{\theta^2} u_{\epsi, \cC}(\theta) + c_2\vert^2}{2} + \epsi U(\theta) \\
&= \frac{c_1^2}{2} + \frac{c_2^2}{2} + \left\langle \cC, \nabla u_{\epsi, \cC}(\theta) \right\rangle + \epsi U(\theta) +  \frac{\left(\partial_{\theta^1} u_{\epsi, \cC}(\theta) \right)^2}{2} + \frac{\left(\partial_{\theta^2} u_{\epsi, \cC}(\theta) \right)^2}{2} \,.
\end{align*}
and it holds that
\begin{equation*}
u_{\epsi, \cC} = u_{\cC}^{(0)} + \epsi u_{\cC}^{(1)} + \mathcal{O}_\cC(\epsi^2)
\end{equation*}
with $u_{\cC}^{(0)} = u_{0,\cC}$.
Since $H_0(\theta,I)$ is integrable (and written in action-angle coordinates), one can choose $u_{0,\cC} \equiv 0$. By { \eqref{eq:alpha1} in Proposition \ref{prop:gomes} (see also \cite{gomes})}  we have  $\alpha^{(1)}(\cC) = [U]_0$, where
\begin{equation*} 
[U]_0 = \int_{\T^2} U(x^1, x^2) \, \D x^1 \wedge \D x^2 \,.
\end{equation*}
Since the sequence $(\epsi_k)_{k \in \N}$ from Assumption \eqref{itm:P} converges to zero, we compare coefficients 
and establish the first order equation
\begin{equation} \label{eq:firstorder1}
[U]_0 =\alpha_\epsi^{(1)}(\cC) =  \braket{\cC, \nabla u^{(1)}_{\epsi, \cC}(\theta)} + U(\theta)\,.
\end{equation}

Averaging \eqref{eq:firstorder1} over the trajectory $\theta(t) = \theta_0 + \omf t \in \T^2$, with initial position $\theta_0 \in \T^2$ and where $\omf = \cC$ is chosen according to \eqref{eq:omega1}, such that the period $T_{\omf}$ satisfies $T_\omf \cdot \omf =  (n,m)$, we get
\begin{equation} \label{eq:average1}
[U]_0 =   \frac{1}{T_\omf} \int_{0}^{T_\omf} \frac{\D}{\D t}u^{(1)}_{\epsi, \cC}(\theta(t)) \, \D t + \frac{1}{T_\omf} \int_{0}^{T_\omf} U(\theta(t)) \, \D t\,.
\end{equation}
The first integral vanishes since $\theta(0)  =\theta(T_\omf) $ such that we are left with
\begin{equation} \label{eq:step2_1}
\int_{0}^{1} \left(U(\theta_0^1 + n t,\theta_0^2 + m t) - [U]_0\right) \D t = 0  
\end{equation}
for all $\theta_0 = (\theta_0^1, \theta_0^2) \in \T^2$, which easily follows from \eqref{eq:average1} after a change of variables. 

Before continuing with the third and final step, we have two important observation: First, by replacing $U \to U - [U]_0$, we can assume w.l.o.g.~that $[U]_0 = 0$. Second, we define the \emph{separable part}, $U_{\rm sep}$, of $U$ as 
\begin{equation} \label{eq:Usep}
U_{\rm sep}(x^1, x^2) := \sum_{(k_1, k_2) \in \mathcal{S}_U \setminus \mathcal{S}_{U,0}} U_{k_1, k_2} \E^{\I 2 \pi k_1 x^1} \E^{\I 2 \pi k_2 x^2}
\end{equation}
(recall the definition of the spectrum and the non-singular spectrum in \eqref{eq:spectrum} and \eqref{eq:nonsinspec}). Then, after a simple computation, we find that 
\begin{equation*}
\int_{0}^{1} U_{\rm sep}(\theta_0^1 + nt, \theta_0^2 + m t) \D t = [U_{\rm sep}]_0\,, \qquad \forall (\theta_0^1, \theta_0^2) \in \T^2\,,
\end{equation*}
holds \emph{generally} (i.e.~independent of the first order relation \eqref{eq:firstorder1}) by means of \eqref{eq:alpha1} in  Proposition~\ref{prop:gomes} (see also Remark \ref{rmk:propertyP1}). We can thus split off the separable part and assume that $\mathcal{S}_U = \mathcal{S}_{U,0}$ in the following. Hence, the third step consists of showing that $\mathcal{S}_U = \mathcal{S}_{U,0} = \emptyset$. 
\\[3mm]
\textbf{\underline{Step (iii).}} The goal of this final step is to establish the following lemma. 
\begin{lem} \label{lem:annihilation} Let $(n,m) \in \mathcal{B}_0(\mathcal{S}_U^\perp)$ as in \eqref{eq:omega1} from {\rm \bf Step (ii)}. Then $U_{jm, -jn} = 0$ for all $j \in \Z \setminus \{ 0 \}$.
\end{lem}
Since $(n,m) \in \mathcal{B}_0(\mathcal{S}_U^\perp)$ were arbitrary,  this proves that
\begin{equation*}
\mathcal{S}_U \subset\left(\{ 0 \} \times \Z\right) \cup \left(\Z \times \{ 0 \}\right)\,,
\end{equation*}
or equivalently $\mathcal{S}_{U,0} = \emptyset$ and we have shown Theorem \ref{thm:1}. It remains to prove Lemma \ref{lem:annihilation}. 
\begin{proof}[Proof of Lemma \ref{lem:annihilation}]
Starting from \eqref{eq:step2_1} we perform a Fourier decomposition to infer
\begin{align*} 
\sum_{ k_1, k_2 \neq 0} \left[U_{k_1, k_2} \int_{0}^{1} \E^{\I 2 \pi k_1 n t} \E^{\I 2 \pi k_2 m t} \D t\right] \E^{\I 2 \pi k_1 \theta_0^1} \E^{\I 2 \pi k_1 \theta_0^2} = 0 \qquad \forall (\theta_0^1, \theta_0^2) \in \T^2\,,
\end{align*}
which implies that 
\begin{equation*}
U_{k_1, k_2} \cdot \delta_{k_1n+ k_2m, \, 0} = 0\,. \qedhere
\end{equation*}
\end{proof} 
Applying Lemma \ref{lem:annihilation} for every $(n,m) \in \mathcal{B}_0(\mathcal{S}_U^\perp)$, we find that $\mathcal{S}_{U,0} = \emptyset$, which finishes the proof of Theorem \ref{thm:1}. \qed

\subsection{Proof of Theorem \ref{thm:2}} \label{subsec:proofthm2}
For notational simplicity, we write $\mu \equiv \mu_2 > 0$ and $V \equiv V_2 \in C^2(\T)$. \\[2mm]
{\bf \underline{Step (i).}} We fix an energy $e >0$ and consider the region of the phase space, where the subsystem in the second pair of coordinates is rotating, i.e.
\begin{equation*}
\frac{p_2^2}{2} - \mu V(x^2) = e^{(2)}>0
\end{equation*}
and for $\frac{p_1^2}{2} = e^{(1)} > 0$ we have $e = e^{(1)} + e^{(2)}$. In a neighborhood of each of the two Liouville tori characterized by $H_0 = e$ and $\frac{p_1^2}{2} = e^{(1)}$ we can find a change of variables $(x^2,p_2) = \Phi^{(2)}_\mu(\theta^2,I_2)$ (and we denote $(x^1, p_1) = (\theta^1, I_1)$) such that the Hamiltonian function $H_0$ gets transformed in action-angle coordinates (see \eqref{eq:actionanglecoordinates}), i.e.
\begin{equation*}
H_0(\theta^1, I_1, \Phi^{(2)}_\mu(\theta^2,I_2)) = \frac{I_1^2}{2} + h^{(2)}_\mu(I_2)
\end{equation*}
for some {smooth function $h^{(2)}_\mu$ agreeing with Mather's $\alpha$-function for the one-dimensional subsystem described by the Hamiltonian $\frac{p_2^2}{2} - \mu V(x^2)$ (see Appendix \ref{app:weakKAM}).}
The change in the order of the four arguments of $H_0$ should not lead to confusion. Now, the perturbed Hamiltonian takes the form
\begin{equation*}
H_\epsi(\theta^1, I_1, \Phi^{(2)}_\mu(\theta^2,I_2)) = \frac{I_1^2}{2} + h^{(2)}_\mu(I_2) + \epsi U(\theta^1, x^2(\theta^2,I_2,\mu))\,,
\end{equation*}
where we write $x^2_\mu(\theta^2,I_2)$ for the first component of $\Phi^{(2)}_\mu(\theta^2,I_2)$. 
\\[3mm]
\textbf{\underline{Step (ii).}} Assume w.l.o.g.~that the $2$-degree ${\deg}^{(2)}_U$ of $U$ is at least $1$ (recall \eqref{eq:A2}), as otherwise we had $U(x) = U_1(x^1)$ and Theorem \ref{thm:2} was proven. Then, for any $(n,m) \in \mathcal{B}_0(\mathcal{S}_U^\perp)$, in particular with $| n| \le  {\deg}_U^{(2)}$, we can find (in the isoenergy manifold $T_{e_\epsi}$ with energy $ e= e_\epsi$ and $\epsi = \epsi_k$ for some $k \in \N$) a rational invariant Liouville torus with rotation vector $\boldsymbol{\omega} = (\omega_1, \omega_2)$, which satisfies 
\begin{equation} \label{eq:omega2}
\frac{\omega_1}{\omega_2} = \frac{n}{m} \in \mathbb{Q} \quad \text{and} \quad \boldsymbol{\omega} = (c_1, \nabla h^{(2)}_\mu(c_2))
\end{equation}
for some $\boldsymbol{c} \in H^1(\T^2, \R) \cong \R^2$ with $\vert c_2 \vert > \gamma + \sqrt{\mu} \mathfrak{c}(V)$ for some $\gamma = \gamma(e, {\deg}^{(2)}_U) > 0$, which we fix now. 

By Assumption \eqref{itm:P} we have
\begin{equation*}
u_{\epsi, \cC} = u_{\cC}^{(0)} + \epsi u_{\cC}^{(1)} + \mathcal{O}_\cC(\epsi^2)
\end{equation*}
with $u_{\cC}^{(0)} = u_{0,\cC}$ and since $H_0(\theta,I)$ is integrable (and written in action-angle coordinates), one can choose $u_{0,\cC} \equiv 0$.
Therefore, by Assumption \eqref{itm:P} again, we expand the Hamilton Jacobi equation \eqref{eq:HJE} as
\begin{align*}
\alpha_{\epsi}(\cC) &= H_\epsi(\theta, \cC + \nabla u_{\epsi, \cC}(\theta)) \\
&= \frac{\vert \partial_{\theta^1} u_{\epsi, \cC}(\theta) + c_1\vert^2}{2} + h_\mu^{(2)}(\partial_{\theta^2} u_{\epsi, \cC}(\theta) + c_2)+ \epsi U(\theta^1,x^2_\mu(\theta^2, \partial_{\theta^2} u_{\epsi, \cC}(\theta) + c_2)) \\
& = \frac{c_1^2}{2} + h_\mu^{(2)}( c_2) + \epsi\left\langle \left(c_1, \nabla h_\mu^{(2)}( c_2)\right), \nabla u_{\epsi, \cC}^{(1)}(\theta) \right\rangle + \epsi U(\theta^1, x^2_\mu(\theta^2, c_2)) \\
& \hspace{1.5cm} + \mathcal{O}\left( \Vert (\nabla^2 h_\mu^{(2)})\vert_{\{ \vert c_2 \vert > \gamma + \sqrt{\mu} \mathfrak{c}(V) \}} \Vert_{C^0} \epsi^2\right) + \mathcal{O} \left( \Vert (\partial_{I_2} \Phi_\mu^{(2)})\vert_{\{ \vert c_2 \vert > \gamma + \sqrt{\mu} \mathfrak{c}(V) \}} \Vert_{C^0} \epsi^2\right)
\end{align*}
Since $|c_2| > \gamma + \sqrt{\mu} \mathfrak{c}(V)$, both error terms are of the order $\mathcal{O}_\gamma(\epsi^2)$.  

Analogously to the proof of Theorem \ref{thm:1} we thus obtain the first order equation
\begin{equation} \label{eq:firstorder2}
[U]_0 = \left\langle \left(c_1, \nabla h_\mu^{(2)}( c_2)\right), \nabla u_{\epsi, \cC}^{(1)}(\theta) \right\rangle +  U(\theta^1, x^2_\mu(\theta^2, c_2))
\end{equation}
where { the constant $\alpha^{(1)} \equiv [U]_0$ is given in \eqref{eq:alpha1} in Proposition \ref{prop:gomes} (see also \cite{gomes}).}
Just as in the proof of Theorem \ref{thm:1}, after averaging \eqref{eq:firstorder2} over the trajectory $\theta(t) = \theta_0 + \omf t \in \T^2$, with initial position $\theta_0 \in \T^2$ and where $\omf$ is chosen according to \eqref{eq:omega2}, such that the period $T_{\omf}$ satisfies $T_\omf \cdot \omf =  (n,m)$, we find
\begin{align} \label{eq:step2_2}
\int_{0}^{1} \left(U(\theta_0^1 +n t, \, x^2_\mu(\theta_0^2 + m t,c_2)) - [U]_0\right) \D t   = 0 
\end{align}
for all $\theta_0 = (\theta_0^1, \theta_0^2) \in \T^2$. 

Finally, analogously to Section \ref{subsec:proofthm1}, we may assume w.l.o.g.~$[U]_0 = 0$ and observe that
\begin{equation*}
	\int_{0}^{1} U_{\rm sep}(\theta_0^1 + nt,x^2_\mu(\theta_0^2 + m t,c_2)) \D t = [U_{\rm sep}]_0 \qquad \forall (\theta_0^1, \theta_0^2) \in \T^2
\end{equation*}
holds \emph{generally} (i.e.~independent of the first order relation \eqref{eq:firstorder2}) by a simple calculation based on \eqref{eq:alpha1} in  Proposition~\ref{prop:gomes} (see also Remark \ref{rmk:propertyP1}). We can thus split off the separable part $U_{\rm sep}$ of $U$ defined in \eqref{eq:Usep} and assume that $\mathcal{S}_U = \mathcal{S}_{U,0}$ in the following. Hence, the third step consists of showing that $\mathcal{S}_U = \mathcal{S}_{U,0} = \emptyset$. 
\\[3mm]
\textbf{\underline{Step (iii).}} We begin this final step with performing a Fourier decomposition in \eqref{eq:step2_2}, such that we obtain
	\begin{align*} 
	\sum_{ k_1\neq 0}\left[ \sum_{0 \neq |k_2| \le \deg_U^{(2)}} U_{k_1, k_2} \int_{0}^{1} \E^{\I 2 \pi k_1 n t} \E^{\I 2 \pi k_2 x_\mu^2(\theta_0^2 + mt, c_2)} \D t \right] \E^{\I 2 \pi k_1 \theta_0^1}  = 0\,, \qquad \forall (\theta_0^1, \theta_0^2) \in \T^2\,,
\end{align*}
which implies that $\big[\cdots\big] = 0$ for every $k_1 \in \Z \setminus \{ 0 \}$ and $\theta_0^2 \in \T$. 

After having eliminated $\theta_0^1 \in \T$, we now fix some $k_1 \in \Z \setminus \{ 0 \}$ and consider the family of functions $\big(f^{(k_1, \mu)}_{k_2}\big)_{0 \neq |k_2| \le \deg_U^{(2)}}$ in the Hilbert space $L^2(\T)$, where 
\begin{equation} \label{eq:lincomb2}
f_{k_2}^{(k_1, \mu)}: \T \to \C\,, \quad \theta_0^2 \mapsto \sum_{\substack{(n,m) \in \mathcal{B}_0(\mathcal{S}^\perp_U)  \\ \exists 0 \neq |\tilde{k}_2| \le \deg_U^{(2)} : \, k_1n+\tilde{k}_2m = 0}} \int_{0}^{1} \E^{\I 2 \pi k_1 n t} \E^{\I 2 \pi k_2 x_\mu^2(\theta_0^2 + mt, c_2)} \D t\,. 
\end{equation}
Note that the sum in \eqref{eq:lincomb2} is finite by Assumption \eqref{itm:A2} (more precisely, it ranges over at most $2 \cdot\deg_U^{(2)}$ elements from $\mathcal{B}_0(\mathcal{S}_U^\perp)$) and we suppressed the dependence of $|c_2| > \gamma + \sqrt{\mu} \mathfrak{c}(V)$ on $(n,m) \in \mathcal{B}_0(\mathcal{S}_U^\perp)$ from the notation (recall \eqref{eq:omega2}). 

In this way, the problem of proving Theorem \ref{thm:2}, i.e.~justifying $\mathcal{S}_{U,0} = \emptyset$, reduces to a question about linear independence for the family of functions \eqref{eq:lincomb2} in the Hilbert space $L^2(\T)$. Recall that the family $\big(f^{(k_1, \mu)}_{k_2}\big)_{0 \neq |k_2| \le \deg_U^{(2)}}$ being linearly independent is equivalent to the \emph{Gram matrix}
\begin{equation} \label{eq:Gram2}
G^{(k_1, \mu)} = \big(G_{k_2, k'_2}^{(k_1, \mu)}\big)_{0 \neq |k_2|, |k'_2| \le \deg_U^{(2)}} \quad \text{with} \quad G_{k_2, k'_2}^{(k_1, \mu)}:= \big\langle f_{k_2}^{(k_1, \mu)}, f_{k'_2}^{(k_1, \mu)} \big\rangle_{L^2(\T)}
\end{equation}
being of full rank, where $\langle g,h\rangle_{L^2(\T)}$ denotes the standard inner product of $g,h \in L^2(\T)$. 

\begin{lem} \label{lem:continuity} There exists $\tilde{\mu} = \tilde{\mu}(\mathcal{C}_2, \deg_U^{(2)}, e) > 0$ such that for all $\mu \in [0, \tilde{\mu}]$ the Gram matrix $G^{(k_1, \mu)}$ from \eqref{eq:Gram2} is of full rank. 
\end{lem}
\begin{proof}Using the version of Lemma \ref{lem:pertLemm} for the inverse function, we find that 
	\begin{equation} \label{eq:perturb2}
		\left\Vert \E^{\I 2 \pi k_2 x^2_{\mu}( \, \cdot \,,c_2)} - \E^{\I 2 \pi k_2 \, \cdot} \right\Vert_{C^0} =  \mathcal{O}\bigg(\deg^{(2)}_U\frac{\mu \Vert V \Vert_{C^0}}{h_{\mu}(\gamma + \sqrt{\mu} \mathfrak{c}(V))}\bigg) =: \mathcal{O}\big(\mu\big)
	\end{equation}
 uniformly in $|k_2| \le \deg_U^{(2)}$ and $(n,m) \in \mathcal{B}_0(\mathcal{S}_U^\perp)$. 
 
 With a slight abuse of notation for the error term, the elements $G^{(k_1, \mu)}_{k_2, k'_2}$ of the Gram matrix can thus be computed as 
 \begin{align*}
\int_{0}^{1} \D \theta_0^2 & \left( \bigg[\sum_{(n,m)} \int_{0}^{1} \D t \, \E^{-\I 2 \pi k_1 n t} \left( \E^{-\I 2 \pi k_2 m t} + \mathcal{O}(\mu) \right) \bigg] \E^{-\I 2 \pi k_2 \theta_0^2} \ \times\right.  \\
&\hspace{3cm}\left. \times \ \E^{\I 2 \pi k'_2 \theta_0^2} \bigg[\sum_{(n',m')} \int_{0}^{1} \D t'\E^{\I 2 \pi k_1 n' t'} \left( \E^{\I 2 \pi k'_2 m' t'} + \mathcal{O}(\mu) \right) \bigg]\right)\,, 
 \end{align*}
where the summations over $(n,m)$ and $(n',m')$ are understood as in \eqref{eq:lincomb2}. Using that, for every $(k_1, k_2) \in \mathcal{S}_{U,0}$ there exist exactly two elements from $\mathcal{B}_0(\mathcal{S}_U^\perp)$ (differing by a sign), we can evaluate both brackets $\big[\cdots\big]$ being equal to $2 + \mathcal{O}\big(\deg_U^{(2)} \mu\big)$.

From this we conclude that
\begin{equation*}
G^{(k_1, \mu)}_{k_2, k'_2} = \int_{0}^{1} \D \theta_0^2  \left[2 + \mathcal{O}\big(\deg_U^{(2)} \mu\big)\right]   \E^{\I 2 \pi (k'_2-k_2) \theta_0^2} \left[2 + \mathcal{O}\big(\deg_U^{(2)} \mu\big)\right] = 4 \,  \delta_{k_2, k'_2} + \mathcal{O}\big(\deg_U^{(2)} \mu\big) \,.
\end{equation*}
Therefore, going back to \eqref{eq:perturb2}, we infer the existence of $\tilde{\mu} = \tilde{\mu}(\mathcal{C}_2, \deg_U^{(2)}, e) > 0$ such that for all $\mu \in [0, \tilde{\mu}]$ the Gram matrix $G^{(k_1, \mu)}$ from \eqref{eq:Gram2} is of full rank. 
\end{proof}
Since $k_1 \in \Z\setminus \{ 0\}$ was arbitrary and Lemma \ref{lem:continuity} is independent of $k_1$, this concludes the proof of Theorem \ref{thm:2}~(a). 

For part (b), we note that $\E^{\I 2 \pi k_2 x_\mu^2(\theta_0^2 + mt, c_2)} $ from \eqref{eq:lincomb2} depends analytically on $\mu$ (see Appendix \ref{app:AApendulum}). Therefore, the function $ \mu \mapsto G^{(k_1, \mu)}$ mapping to the Gram matrix \eqref{eq:Gram2}, for every fixed $k_1 \in \Z \setminus \{0\}$, is also analytic.\footnote{Using joint continuity of $(u,\mu) \mapsto \E^{\I 2 \pi k_2 x_\mu^2(u, c_2)}$, it is an elementary exercise to show that the integrals over $t$ and $\theta_0^2$ do not disturb the analyticity in $\mu$.} This in turn implies that  $\det(G^{(k_1, \mu)})$ is analytic in $\mu$ and thus, since $\det(G^{(k_1, \mu)}) \neq 0$ for $\mu \in (0,\tilde{\mu})$ (see Lemma \ref{lem:continuity}), we find that the zero set 
\begin{equation*}
\mathcal{E}_0^{(k_1)} := \{  \mu \in (0,\infty) \mid \det(G^{(k_1, \mu)}) = 0 \} \subset (\tilde{\mu}, \infty)
\end{equation*}
 of $\mu \mapsto \det(G^{(k_1, \mu)})$ is at most countable (finite in every compact subset), i.e.~in particular a set of zero measure. Finally, setting 
 \begin{equation*}
\mathcal{E}_0 := \bigcup_{k_1 \in \Z \setminus \{ 0\}} \mathcal{E}_0^{(k_1)}\,, 
 \end{equation*}
we constructed the exceptional null set, for which the conclusion $\mathcal{S}_{U,0} = \emptyset$ is not valid. 

This finishes the proof of Theorem \ref{thm:2}~(b). \qed

\subsection{Proof of Theorem \ref{thm:3}} \label{subsec:proofthm3}
\textbf{\underline{Step (i).}} We fix an energy $e >0$ and consider the region of the phase space, where both one-dimensional subsystems are rotating, i.e.
\begin{equation*}
\frac{p_1^2}{2} - \mu_1 V_1(x^1) = e^{(1)}>0\quad \text{and} \quad   \frac{p_2^2}{2} - \mu_2 V_2(x^2) = e^{(2)}>0\,,
\end{equation*}
such that we have $e = e^{(1)} + e^{(2)}$. In a neighborhood of each of the two Liouville tori characterized by $H_0 = e$ and $\frac{p_1^2}{2} - \mu_1 V_1(x^1)= e^{(1)}$, we can find two changes of variables $(x^1,p_1) = \Phi^{(1)}_{\mu_1}(\theta^1,I_1)$ and $(x^2,p_2) = \Phi^{(2)}_{\mu_2}(\theta^2,I_2)$  such that the Hamiltonian function $H_0$ gets transformed in action-angle coordinates (see \eqref{eq:actionanglecoordinates}), i.e.
\begin{equation*}
H_0(\Phi^{(1)}_{\mu_1}(\theta^1,I_1), \Phi^{(2)}_{\mu_2}(\theta^2,I_2)) = h^{(1)}_{\mu_1}(I_1) + h^{(2)}_{\mu_2}(I_2)
\end{equation*}
for some smooth functions $h^{(1)}_{\mu_1}$ and $h^{(2)}_{\mu_2}$,which agree with Mather's $\alpha$-functions for the one-dimensional subsystem described by the Hamiltonians $\frac{p_1^2}{2} - \mu_1 V(x^1)$ resp.~$\frac{p_2^2}{2} - \mu_2 V(x^2)$ (see Appendix \ref{app:weakKAM}).
As in the proof of Theorem~\ref{thm:2}, the change in the order of the four arguments of $H_0$ should not lead to confusion. 

Now, the perturbed Hamiltonian takes the form
\begin{equation*}
H_\epsi(\Phi^{(1)}_{\mu_1}(\theta^1,I_1), \Phi^{(2)}_{\mu_2}(\theta^2,I_2)) = h^{(1)}_{\mu_1}(I_1)+ h^{(2)}_{\mu_2}(I_2) + \epsi U(x^1_{\mu_1}(\theta^1,I_1), x^2_{\mu_2}(\theta^2,I_2))\,,
\end{equation*}
where we write $x^i_{\mu_i}(\theta^i,I_i)$ for the first component of $\Phi^{(i)}_{\mu_i}(\theta^i,I_i)$, $i \in \{1,2\}$. 
\\[3mm]
\textbf{\underline{Step (ii).}} 
Analogously to the proof of Theorem \ref{thm:2}, we assume w.l.o.g.~that the $1$- and $2$-degree ${\deg}^{(1)}_U$ and ${\deg}^{(2)}_U$ of $U$ are at least $1$ (recall \eqref{eq:A3}), as otherwise we had $U(x) = U_2(x^2)$ or $U(x) = U_1(x^1)$ and Theorem \ref{thm:3} was proven. Then, for any $(n,m) \in \mathcal{B}_0(\mathcal{S}_U^\perp)$, in particular with $|m| \le  {\deg}_U^{(1)}$ and $|n| \le  {\deg}_U^{(2)}$, we can find (in the isoenergy manifold $T_{e_\epsi}$ with energy $ e= e_\epsi$  and $\epsi = \epsi_k$ for some $k \in \N$) a rational invariant Liouville torus with rotation vector $\boldsymbol{\omega} = (\omega_1, \omega_2)$ which satisfies 
\begin{equation} \label{eq:omega3}
\frac{\omega_1}{\omega_2} = \frac{n}{m} \in \mathbb{Q} \quad \text{and} \quad \boldsymbol{\omega} = (\nabla h^{(1)}_{\mu_1}(c_1), \nabla h^{(2)}_{\mu_2}(c_2))
\end{equation}
for some $\boldsymbol{c} \in H^1(\T^2, \R) \cong \R^2$ with $\vert c_1 \vert > \gamma_1 + \sqrt{\mu_1} \mathfrak{c}(V_1)$ and $\vert c_2 \vert > \gamma_2 + \sqrt{\mu_2} \mathfrak{c}(V_2)$ for some $\gamma_1 = \gamma_1(e, {\deg}^{(1)}_U) > 0$ resp.~$\gamma_2 = \gamma_2(e, {\deg}^{(2)}_U) > 0$, which we fix now. 

By Assumption \eqref{itm:P} we have
\begin{equation*}
u_{\epsi, \cC} = u_{\cC}^{(0)} + \epsi u_{\cC}^{(1)} + O_\cC(\epsi^2)
\end{equation*}
with $u_{\cC}^{(0)} = u_{0,\cC}$ and since $H_0(\theta,I)$ is integrable (and written in action-angle coordinates), one can choose $u_{0,\cC} \equiv 0$.
Therefore, by Assumption \eqref{itm:P} again, we expand the Hamilton Jacobi equation \eqref{eq:HJE} as
\begin{align*}
\alpha_{\epsi}(\cC) &= H_\epsi(\theta, \cC + \nabla u_{\epsi, \cC}(\theta) ) \\
&= h_{\mu_1}^{(1)}(\partial_{\theta^1} u_{\epsi, \cC}(\theta) + c_1) + h_{\mu_2}^{(2)}(\partial_{\theta^2} u_{\epsi, \cC}(\theta) + c_2) \\[1mm] 
& \hspace{2cm}+ \epsi U(x^1_{\mu_1}(\theta^1, \partial_{\theta^1} u_{\epsi, \cC}(\theta) + c_1),x^2_{\mu_2}(\theta^2, \partial_{\theta^2} u_{\epsi, \cC}(\theta) + c_2)) \\[1mm]
& = \sum_{i=1}^{2}h_{\mu_i}^{(i)}( c_i) + \epsi\left\langle \left(\nabla h_{\mu_1}^{(1)}( c_1), \nabla h_{\mu_2}^{(2)}( c_2)\right), \nabla u_{\epsi, \cC}^{(1)}(\theta) \right\rangle + \epsi U(x^1_{\mu_1}(\theta^1, c_1), x^2_{\mu_2}(\theta^2, c_2)) \\
& \hspace{1cm} + \mathcal{O}\bigg( \sum_{i=1}^{2}\big(\Vert (\nabla^2 h_{\mu_i}^{(i)})\vert_{\{ \vert c_i \vert > \gamma_i + \sqrt{\mu_i} \mathfrak{c}(V_i) \}} \Vert_{C^0} + \Vert (\partial_{I_i} \Phi_{\mu_i}^{(i)})\vert_{\{ \vert c_i \vert > \gamma_i + \sqrt{\mu_i} \mathfrak{c}(V_i) \}} \Vert_{C^0} \big)\epsi^2\bigg)
\end{align*}
Since $|c_i| > \gamma_i + \sqrt{\mu_i} \mathfrak{c}(V_i)$, the error term is of order $\mathcal{O}_{\gamma_i}(\epsi^2)$.  

Analogously to the proofs of Theorem \ref{thm:1} and Theorem \ref{thm:2} we thus obtain the first order equation
\begin{equation} \label{eq:firstorder3}
[U]_0 = \left\langle \left(\nabla h_{\mu_1}^{(1)}( c_1), \nabla h_{\mu_2}^{(2)}( c_2)\right), \nabla u_{\epsi, \cC}^{(1)}(\theta) \right\rangle +  U(x^1_{\mu_1}(\theta^1, c_1), x^2_{\mu_2}(\theta^2, c_2))
\end{equation}
where { the constant $\alpha^{(1)} \equiv [U]_0$ is again given by \eqref{eq:alpha1} in Proposition \ref{prop:gomes} (see also \cite{gomes}).}
Just as in the proof of Theorem \ref{thm:1} and Theorem \ref{thm:2}, after averaging \eqref{eq:firstorder3} over the trajectory $\theta(t) = \theta_0 + \omf t \in \T^2$, with initial position $\theta_0 \in \T^2$ and where $\omf$ is chosen according to~\eqref{eq:omega3}, such that the period $T_{\omf}$ satisfies $T_\omf \cdot \omf =  (n,m)$, we find
\begin{equation} \label{eq:step2_3}
	\int_{0}^{1} \left(U(x^1_{\mu_1}(\theta_0^1 +n t, c_1), \, x^2_{\mu_2}(\theta_0^2 + m t,c_2)) - [U]_0  \right) \, \D t = 0
\end{equation}
for all $\theta_0 = (\theta_0^1, \theta_0^2) \in \T^2$. 

Finally, analogously to Section \ref{subsec:proofthm1} and Section \ref{subsec:proofthm2}, we may assume w.l.o.g.~$[U]_0 = 0$ and observe that
\begin{equation*}
	\int_{0}^{1} U_{\rm sep}(x^1_{\mu_1}(\theta_0^1 +n t, c_1), \, x^2_{\mu_2}(\theta_0^2 + m t,c_2)) \D t = [U_{\rm sep}]_0 \qquad \forall (\theta_0^1, \theta_0^2) \in \T^2
\end{equation*}
holds \emph{generally} (i.e.~independent of the first order relation \eqref{eq:firstorder3}) by a simple calculation based on \eqref{eq:alpha1} in  Proposition~\ref{prop:gomes} (see also Remark \ref{rmk:propertyP1}). We can thus split off the separable part $U_{\rm sep}$ of $U$ defined in \eqref{eq:Usep} and assume that $\mathcal{S}_U = \mathcal{S}_{U,0}$ in the following. Hence, the third step consists of showing that $\mathcal{S}_U = \mathcal{S}_{U,0} = \emptyset$. 
\\[3mm]
\textbf{\underline{Step (iii).}} We begin this final step with performing a Fourier decomposition in \eqref{eq:step2_3}, such that we obtain
\begin{align*} 
	\sum_{\substack{ 0 \neq |k_1| \le \deg_U^{(1)} \\ 0 \neq |k_2| \le \deg_U^{(2)}}} U_{k_1, k_2} \int_{0}^{1} \E^{\I 2 \pi k_1 x_{\mu_1}^1(\theta_0^1 + nt, c_1)} \E^{\I 2 \pi k_2 x_{\mu_2}^2(\theta_0^2 + mt, c_2)} \D t   = 0\,, \qquad \forall (\theta_0^1, \theta_0^2) \in \T^2\,.
\end{align*}
Analogously to the proof of Theorem \ref{thm:2}, we now consider the family of functions 
\begin{equation*}
\big(f^{(\mu_1, \mu_2)}_{k_1, k_2}\big)_{0 \neq |k_1| \le \deg_U^{(1)}, 0 \neq |k_2| \le \deg_U^{(2)}}
\end{equation*}
 in the Hilbert space $L^2(\T^2)$, where 
\begin{equation} \label{eq:lincomb3}
	f_{k_1, k_2}^{(\mu_1, \mu_2)}: \T^2 \to \C\,, \quad (\theta_0^1, \theta_0^2) \mapsto \sum_{\substack{(n,m) \in \mathcal{B}_0(\mathcal{S}^\perp_U)  }} \int_{0}^{1} \E^{\I 2 \pi k_1 x_{\mu_1}^1(\theta_0^1 + nt, c_1)}  \E^{\I 2 \pi k_2 x_\mu^2(\theta_0^2 + mt, c_2)} \D t\,. 
\end{equation}
Note that the sum in \eqref{eq:lincomb3} is finite by Assumption \eqref{itm:A3} (more precisely, it ranges over the at most $(2 \deg_U^{(1)}) \cdot (2 \deg_U^{(2)}) $ elements from $\mathcal{B}_0(\mathcal{S}_U^\perp)$) and we suppressed the dependence of $|c_i| > \gamma_i + \sqrt{\mu_i} \mathfrak{c}(V_i)$ on $(n,m) \in \mathcal{B}_0(\mathcal{S}_U^\perp)$ from the notation (recall \eqref{eq:omega3}). 

In this way, the problem of proving Theorem \ref{thm:3}, i.e.~justifying $\mathcal{S}_{U,0} = \emptyset$, reduces to a question about linear independence for the family of functions \eqref{eq:lincomb3} in the Hilbert space $L^2(\T^2)$. Recall that the family $\big(f^{( \mu)}_{k_1, k_2}\big)_{(k_1, k_2)}$ being linearly independent is equivalent to the \emph{Gram matrix} $G^{( \mu)} $ with entries 
\begin{equation} \label{eq:Gram3}
 G_{(k_1, k_2), (k'_1, k'_2)}^{(\mu_1, \mu_2)}:= \big\langle f_{k_1, k_2}^{( \mu_1, \mu_2)}, f_{k'_1, k'_2}^{(\mu_1, \mu_2)} \big\rangle_{L^2(\T^2)} \quad \text{for} \quad 0 \neq |k_i|, |k'_i| \le \deg_U^{(i)}\,, \ i \in \{1,2\}\,, 
\end{equation}
being of full rank, where $\langle g,h\rangle_{L^2(\T^2)}$ denotes the standard inner product of $g,h \in L^2(\T^2)$. 

\begin{lem} \label{lem:continuity2} There exist $\tilde{\mu}_i = \tilde{\mu}(\mathcal{C}_i,\deg_U^{(1)},  \deg_U^{(2)}, e) > 0$ such that for all $\mu_i \in [0, \tilde{\mu}_i]$, $i \in \{1,2\}$, the Gram matrix $G^{(\mu_1, \mu_2)}$ from \eqref{eq:Gram3} is of full rank. 
\end{lem}
\begin{proof}Using the version of Lemma \ref{lem:pertLemm} for the inverse function, we find that 
	\begin{equation} \label{eq:perturb3}
		\left\Vert \E^{\I 2 \pi k_i x^i_{\mu_i}( \, \cdot \,,c_i)} - \E^{\I 2 \pi k_i \, \cdot} \right\Vert_{C^0} =  \mathcal{O}\bigg(\deg^{(i)}_U\frac{\mu_i \Vert V_i \Vert_{C^0}}{h_{\mu_i}^{(i)}(\gamma_i + \sqrt{\mu_i} \mathfrak{c}(V_i))}\bigg) =: \mathcal{O}\big(\mu_i\big)
	\end{equation}
	uniformly in $|k_i| \le \deg_U^{(i)}$ and $(n,m) \in \mathcal{B}_0(\mathcal{S}_U^\perp)$. 
	
	Similarly to Lemma \ref{lem:continuity}, with a slight abuse of notation for the error term, the elements $ G_{(k_1, k_2), (k'_1, k'_2)}^{(\mu_1, \mu_2)}$ of the Gram matrix can thus be computed as 
	\begin{align*}
			\int_{0}^{1} \D \theta_0^1	\int_{0}^{1} \D \theta_0^2 & \left( \bigg[\sum_{(n,m)} \int_{0}^{1} \D t \, \left( \E^{-\I 2 \pi k_1 n t} + \mathcal{O}(\mu_1) \right) \left( \E^{-\I 2 \pi k_2 m t} + \mathcal{O}(\mu_2) \right) \bigg] \E^{-\I 2 \pi k_1 \theta_0^1}\E^{-\I 2 \pi k_2 \theta_0^2} \ \times\right.  \\
		&\left. \times \ \E^{\I 2 \pi k'_1 \theta_0^1}\E^{\I 2 \pi k'_2 \theta_0^2} \bigg[\sum_{(n',m')} \int_{0}^{1} \D t' \, \left( \E^{\I 2 \pi k'_1 n' t'} + \mathcal{O}(\mu_1) \right) \left( \E^{\I 2 \pi k'_2 m' t'} + \mathcal{O}(\mu_2) \right) \bigg]\right)\,, 
	\end{align*}
	where the summations over $(n,m)$ and $(n',m')$ are understood as in \eqref{eq:lincomb3}. Using that for every $(k_1, k_2) \in \mathcal{S}_{U,0}$ there exist exactly two elements from $\mathcal{B}_0(\mathcal{S}_U^\perp)$ (differing by a sign), we can evaluate both brackets $\big[\cdots\big]$ being given by
	\begin{equation*}
2 + \mathcal{O}\big(\deg_U^{(1)} \deg_U^{(2)}\mu_1\big)+ \mathcal{O}\big(\deg_U^{(1)}\deg_U^{(2)} \mu_2\big) =: 2 +  \mathcal{O}\big( \deg_U^{(1)} \deg_U^{(2)} (\mu_1 + \mu_2 )\big)\,, 
	\end{equation*}
after absorption of the second order error in the first order ones. 
	
	From this we conclude that
	\begin{align*}
 G_{(k_1, k_2), (k'_1, k'_2)}^{(\mu_1, \mu_2)} 
&= \int_{0}^{1} \D \theta_0^1  \int_{0}^{1} \D \theta_0^2  \bigg( \left[ 2 +  \mathcal{O}\big( \deg_U^{(1)} \deg_U^{(2)} (\mu_1 + \mu_2 )\big) \right]   \E^{\I 2 \pi (k'_1-k_1)\theta_0^1} \, \times \\
& \hspace{2.5cm} \times \,  \E^{\I 2 \pi (k'_2-k_2) \theta_0^2} \left[ 2 +  \mathcal{O}\big( \deg_U^{(1)} \deg_U^{(2)} (\mu_1 + \mu_2 )\big) \right] \bigg) \\
&= 4 \,  \delta_{k_1, k'_1} \,  \delta_{k_2, k'_2}+\mathcal{O}\big( \deg_U^{(1)} \deg_U^{(2)} (\mu_1 + \mu_2 )\big) \,.
	\end{align*}
	Therefore, going back to \eqref{eq:perturb3}, we infer the existence of $\tilde{\mu}_i = \tilde{\mu}(\mathcal{C}_i,\deg_U^{(1)},  \deg_U^{(2)}, e) > 0$, $i \in \{1,2\}$, such that for all $\mu_i \in [0, \tilde{\mu}_i]$ the Gram matrix $G^{(\mu_1, \mu_2)}$ from \eqref{eq:Gram3} is of full rank. 
\end{proof}
This finishes the proof of Theorem \ref{thm:3}~(a). For part (b), similarly to the proof of Theorem~\ref{thm:2}~(b), we observe that for every fixed $\mu_1 \in [0,\tilde{\mu}_1]$ the function $\mu_2 \mapsto \det\big(G^{(\mu_1, \mu_2)}\big)$ is analytic. Since $\det(G^{(\mu_1, \mu_2)}) \neq 0$ for $\mu_2 \in (0,\tilde{\mu}_2)$ (see Lemma \ref{lem:continuity2}), we find that the zero set 
\begin{equation*}
	\mathcal{E}_0^{(\mu_1)} := \{  \mu_2 \in (0,\infty) \mid \det(G^{(\mu_1, \mu_2)}) = 0 \} \subset (\tilde{\mu}_2, \infty)
\end{equation*}
of $\mu_2 \mapsto \det(G^{(\mu_1, \mu_2)})$ is at most countable (finite in every compact subset), i.e.~in particular a (one-dimensional) set of zero measure.

Finally, for part (c), we note that, similarly to the proof of Theorem \ref{thm:2}~(b) and by means of Hartogs's theorem on separate analyticity, the function  $(\mu_1, \mu_2) \mapsto \det\big(G^{(\mu_1, \mu_2)}\big)$ is (jointly) analytic. 
 Since $\det(G^{(\mu_1, \mu_2)}) \neq 0$ for $(\mu_1, \mu_2) \in (0,\tilde{\mu}_1) \times (0,\tilde{\mu}_2)$ (see Lemma \ref{lem:continuity2}), we find that the zero set 
\begin{equation*}
	\mathcal{E}_0 := \{  (\mu_1, \mu_2) \in (0,\infty)\times (0,\infty) \mid \det(G^{(\mu_1, \mu_2)}) = 0 \} \subset (\tilde{\mu}_1, \infty) \times (\tilde{\mu}_2, \infty)
\end{equation*}
of $(\mu_1, \mu_2) \mapsto \det(G^{(\mu_1, \mu_2)})$ is a (two-dimensional) set of zero measure. 

This concludes the proof of Theorem \ref{thm:3}~(c). \qed

\section{Concluding remarks and outlook} \label{sec:outlook}
We have shown that integrable deformations of Liouville metrics on $\T^2$ are Liouville metrics -- at least when more restrictive conditions on the unperturbed metric are balanced with more general conditions on the perturbation. Removing this balancing, i.e.~showing that \emph{arbitrary} integrable deformations of \emph{arbitrary} Liouville metrics remain of Liouville type, is an interesting problem for future investigations resolving the conjecture proposed at the end of Section \ref{sec:Mainresults}. This would require stronger versions of Lemmas \ref{lem:continuity} and~\ref{lem:continuity2} in two senses: 
\begin{itemize}
\item[(a)] Allow for possibly \emph{infinitely many non-zero Fourier coefficients} and refrain from restricting to trigonometric polynomials. A resolution of this issue has been found in the context of the \emph{perturbative Birkhoff conjecture} \cite{ADK, KS} concerning integrable billiards. Here, the authors studied the matrix of correlations between the standard basis $(\E^{\I 2 \pi k x})_{k \in \Z}$ of $L^2(\T)$ and certain deformed dynamical modes (given as some kind of Jacobi elliptic function, see Appendix \ref{app:AApendulum}), corresponding to $\E^{\I 2 \pi k_i x^i_{\mu_i}( \, \cdot \,,c_i)}$ in Lemma~\ref{lem:continuity} and Lemma~\ref{lem:continuity2}.  Exponential estimates for the entries of this matrix (obtained from considering the maximal width of a strip of analyticity around the real axis for the dynamical modes), allowed to prove a suitable full-rank lemmas, also for infinitely many coefficients. 
\item[(b)] Allow \emph{arbitrary} $\tilde{\mu_i}> 0$ and refrain from restricting to small ones. Also for this issue, a potential resolution might be found by analytically extending action-angle coordinates to the complex plane and exploiting their singularities away from the real axis. However, this requires the potentials $V_i$ in the unperturbed Hamiltonian to be restrictions of holomorphic functions and as such way more special than generic $V_i \in C^2(\T)$. 
\end{itemize}

Moreover, we note that, in \cite{KS} the authors also outlined a potential strategy for proving the classical (non-perturbative) Birkhoff conjecture, which might possibly be adapted for proving a suitably weakened version of the folklore conjecture given in Section \ref{sec:integrablemetricsonT2}.

We end this section with a brief list of open problems being related to the main results of the present paper: 
\begin{itemize}
\item[(i)]	As described above, it is an natural follow-up problem to extend our results to the situation, where \emph{arbitrary} integrable deformations of \emph{arbitrary} Liouville metrics remain of Liouville type, i.e.~remove the restricting assumptions from \eqref{itm:A1} - \eqref{itm:A3} and prove the conjecture formulated at the end of Section \ref{sec:Mainresults}. 
\item[(ii)] In particular, starting with (the time-independent version of) Arnold's example \cite{arnolddiff} for diffusion,
 \begin{equation*}
H_0(x,p) = \frac{p_1^2}{2} + \frac{p_2^2}{2} - \mu \big(1 - \cos(2 \pi x^2)\big)\,,
 \end{equation*}
is it possible do deduce rigidity, similarly to Theorem \ref{thm:2}, but without restricting to the perturbation being a trigonometric polynomial in $x^2$ and any smallness condition on $\mu \in [0,1]$? In this case, the \emph{full rank lemma} might be obtained by proving non-degeneracy of certain infinite-dimensional matrices, which have Fourier coefficients of powers of Jacobi elliptic functions (see Appendix \ref{app:AApendulum}) as their entries. 
\item[(iii)]In view of the potential counterexample constructed in \cite{corsikaloshin}, it is a major task to completely settle the \emph{Folklore Conjecture} mentioned in Sections \ref{sec:Introduction} and \ref{sec:integrablemetricsonT2}, i.e.~clarify which part is only 'folklore' and which part is 'real'. 
\item[(iv)] In particular, the main result of \cite{corsikaloshin}, which we stated in Theorem \ref{thm:corsikaloshin}, should be extended to showing that the system is really integrable on an open set in the phase space and not only on an isoenergy manifold. 
\item[(v)] For our main results, we assumed the preservation of rational invariant tori 'outside the eye of the pendulum' (cf.~Figure \ref{fig:1}). Can one obtain the same result, if only tori 'inside the eye' are preserved? 
\item[(vi)] An alternative approach to the one chosen here, could be to study perturbations of the additional first integral \eqref{eq:first int}, i.e.~write $F_\epsi = F_0 + \epsi F_1 + \mathcal{O}(\epsi^2)$ and use the vanishing of the Poisson bracket $\{H_\epsi, F_\epsi\} = 0$ with $H_\epsi = H_0 + \epsi U$ to obtain the first-order equation
\begin{equation*}
\{H_0, F_1\} + \{U, F_0\} = 0
\end{equation*}
for the perturbing potential $U$. 
\item[(vii)] Does there exist a Riemannian metric $g$ on $\T^2$, such that its geodesic flow admits hyperbolic periodic orbits of \emph{at least three} different homotopy types? If yes, does there exist a Liouville metric with this property?\footnote{These questions were suggested by Vadim Kaloshin.} 
\end{itemize} 
\noindent\textit{Acknowledgments.}
I am very grateful to Vadim Kaloshin for suggesting the topic, his guidance during this project, and many helpful comments on an earlier version of the manuscript. Moreover, I would like to thank Comlan Edmond Koudjinan and Volodymyr Riabov for interesting discussions. Partial financial support by the ERC Advanced Grant `RMTBeyond' No.~101020331 is gratefully acknowledged. 
\appendix
\addtocontents{toc}{\protect\setcounter{tocdepth}{-1}}
\section{Generalization to higher dimensions} \label{app:higher dim}
Our results from Section \ref{sec:Mainresults} immediately generalize to higher dimensions $d \ge 3$. In this setting, we define the Hamiltonian function
\begin{equation} \label{eq:H0 app}
	H_0(x,p) = \sum_{i=1}^{d}\left(\frac{p_i^2}{2}  - \mu_i\, V_i(x^i) \right)
\end{equation}
on $T^*\mathbb{T}^d$, where $\mu_i \in [0,\infty)$ are parameters, and $V_i \in C^2(\mathbb{T})$ with $\Vert V_i \Vert \le \mathcal{C}_i$, $V_i \ge 0$ are Morse functions (or constant). We may assume w.l.o.g.~that $\min_{x^i} V_i(x^i) = 0$. The system \eqref{eq:H0 app} is clearly integrable, since additional first integrals can easily be found as 
\begin{equation} \nonumber
F_i(x,p) = \frac{p_i^2}{2} - \mu_i \, V_i(x^i)\,, \qquad i \in \set{1, ... , d-1}\,. 
\end{equation}
Completely analogous to Section \ref{sec:Mainresults}, we perturb the integrable system \eqref{eq:H0 app} as $H_\epsi = H_0 + \epsi U$ with $\epsi \in \R$ by an additive potential $U \in C^2(\T^d)$, which we assume to have an absolutely convergent Fourier series. 

Now, the analogs of the assumptions in Section \ref{sec:Mainresults} read as follows. 
\\[5mm]
\textbf{1.~Assumptions on the perturbed Hamiltonian function $H_\epsi$.} \\* Let $H_0 \in C^2(T^*\T^d)$ denote the Hamiltonian function from \eqref{eq:H0 app} with $\Vert V_i \Vert \le \mathcal{C}_i$ and $\mu_i \in [0,\tilde{\mu}_i]$ for some $\tilde{\mu_i} \in [0,\infty)$, $i \in \{1,...,d\}$, and $U \in C^2(\T^d)$ be a perturbing potential, which satisfies the following assumption. 
\begin{itemize}
\item[\bf(\mylabel{itm:A4}{A4})] \label{itm:A4} If $\tilde{\mu}_i = 0$ for the first $0 \le d_{\rm flat} \le d$ indices, there exist $d^{(k)} \ge 0$ for $k \in \set{d_{\rm flat} + 1, ... , d}$ such that 
	\begin{equation} \label{eq:A4}
\mathcal{S}_U \subset \Z^{d_{\rm flat}} \times \left( [-d^{(d_{\rm flat} + 1)}, d^{(d_{\rm flat} + 1)}] \times \cdots \times [-d^{(d)}, d^{(d)}]\right)
	\end{equation}
i.e.~$U \in C^2(\mathbb{T}^d)$ is a trigonometric polynomial in the last $(d - d_{\rm flat})$ variables. 
\end{itemize}
As in Section \ref{sec:Mainresults}, we denote the minimum over all $d^{(i)}$ such that \eqref{eq:A4}  holds as $\deg_U^{(i)}$ and call it the \textit{$i$-degree of $U$}. 

Note that Proposition \ref{prop:KAM} immediately generalizes to higher dimensions, such that we can formulate the analog of Assumption \eqref{itm:P} as follows. 
\\[5mm]
\noindent\textbf{2.~Assumptions on the preserved integrability of $H_\epsi$.} \\*
Let $H_0 \in C^2(T^*\T^d)$ denote the Hamiltonian function from \eqref{eq:H0 app} satisfying Assumption~\eqref{itm:A4}, and $U$ a perturbing potential, such that the following statement concerning the perturbed Hamilton-Jacobi equation (HJE)
\begin{equation} \label{eq:HJE app}
	\alpha_{\epsi}(\cC) = H_\epsi(x, \cC + \nabla_x u_{\epsi, \cC}(x) )
\end{equation}
as well as the preserved integrability of $H_\epsi$ is satisfied. 

\begin{itemize} 
\item[\bf(\mylabel{itm:P'}{P'})] \label{itm:P'}  There exists an energy $e >0$, such that for every $\boldsymbol{b} \in \mathcal{B}_0(\mathcal{S}_U^\perp)$ (recall \eqref{eq:basis0perp}) there exists a sequence $(\epsi_k)_{k \in \N}$ with $\epsi_k \neq 0$ but $\epsi_k \to 0$ such that for any $\mu_i \in [0,\tilde{\mu}_i]$ we have the following:  
	\begin{itemize}
		\item[(i)] The $\boldsymbol{b}$-torus from (the analog of) Proposition~\ref{prop:KAM}  characterized by $\boldsymbol{c} \in H^1(\T^d, \R) \cong \R^d$ with 
		\begin{equation} \label{eq:ccond app}
			|c_i|> \sqrt{\mu_i}\, \mathfrak{c}(V_i)
		\end{equation}
		in the isoenergy submanifold $T_e$ is preserved under the sequence of deformations $(H_{\epsi_k})_{k \in \N}$, where $\mathfrak{c}(V_i)$ is defined in \eqref{eq:defc(V)}. 
		\item[(ii)] For $\cC \in H^1(\T^d, \R)$ satisfying \eqref{eq:ccond app}, Mather's $\alpha$-function and a solution $u_{\epsi,\cC}:\T^d \to \mathbb{R}$ of the HJE \eqref{eq:HJE app} can be expanded to first order in $\epsi$, i.e. 
		\begin{equation} \nonumber
			u_{\epsi, \cC} = u_{\cC}^{(0)} + \epsi u_{ \cC}^{(1)} + \mathcal{O}_\cC(\epsi^2)\quad \text{and} \quad \alpha_{\epsi} = \alpha^{(0)} + \epsi \alpha^{(1)} + \mathcal{O}(\epsi^2)\,,
		\end{equation}  
		where $u_{\cC}^{(0)} ,\, u_{ \cC}^{(1)} \in C^{1,1}(\T^d)$  and $O_\cC(\epsi^2)$ is understood in $C^{1,1}$-sense.
	\end{itemize}
\end{itemize}
We can now formulate our generalized main result. 
\begin{theorem} \label{thm:4}
Let $H_\epsi$ satisfy Assumption \eqref{itm:A4} and Assumption \eqref{itm:P'} for some energy $e > 0$.  If $V_j$ is analytic for $j \in \set{d - d_{\rm anlyt} + 1, ... , d}$, where $0 \le d_{\rm anlyt} \le d - d_{\rm flat}$, and $\tilde{\mu}_k = \tilde{\mu}_k(\mathcal{C}_k, \deg_U^{(d_{\rm flat}+1)},..., \deg_U^{(d)}, e) > 0$ for $k \in \set{d_{\rm flat} + 1, ... , d-d_{\rm anlyt}}$ are small enough, then $U$ is separable, i.e. there exist $U_1, ... , U_d \in C^2(\T)$ such that
\begin{equation*}
U(x^1, ... , x^d) = U_1(x^1) + ... + U_d(x^d) \qquad \forall (x^1, ... , x^d) \in \T^d\,. 
\end{equation*} 
This is irrespective of $\tilde{\mu}_j > 0$ for $j \in \set{d - d_{\rm anlyt} + 1, ... , d}$, but only for 
\begin{equation*}
(\mu_{d-d_{\rm anlyt}+1}, ... , \mu_d) \in [0, \tilde{\mu}_{d-d_{\rm anlyt}+1}] \times ... \times [0, \tilde{\mu}_{d}]
\end{equation*}
outside of an exceptional $d_{\rm anlyt}$-dimensional null-set (depending on $(\mu_{d_{\rm flat} + 1} , ... , \mu_{d-d_{\rm anylt}}) $). 
\end{theorem}

\section{Basic perturbation lemma} \label{app:pertLemm}
In this appendix, we state and prove a basic perturbation lemma, which is instrumental in the continuity arguments required for the proofs of Lemma \ref{lem:continuity} and Lemma \ref{lem:continuity2}. 
\begin{lem} \label{lem:pertLemm} Let $V \in C^1(\T)$ be a non-negative function, $\mu \in [0, 1]$, and define the Hamiltonian function
	\begin{equation} \label{eq:pertHamil}
H_\mu(p,x) = \frac{p^2}{2} - \mu V(x)
	\end{equation}
	on the cotangent bundle $T^*\T$. In the neighborhood of a fixed energy $E > 0$, we can find {\rm action-angle coordinates} $(I, \theta)$ of \eqref{eq:pertHamil} as
	\begin{equation} \label{eq:aacoord}
I = \pm \int_{0}^{1} \sqrt{2(E+ \mu V(x))} \, \D x\,, \qquad \theta = \pm  \frac{\int_{0}^{x} \frac{\D x'}{\sqrt{1+\mu V(x')/E}}}{\int_{0}^{1} \frac{\D x'}{\sqrt{1+\mu V(x')/E}}}\,.
	\end{equation}
	Regarding $\theta = \theta(x)$ as a function on $\T$, we have $\theta \in C^1(\T)$ and 
	\begin{equation} \label{eq:pertLemm}
\Vert \theta \mp x \Vert_{C^1} = \mathcal{O}\left(\frac{\mu \Vert V \Vert_{C^0}}{E}\right) \quad \text{as} \quad \mu \to 0\,.
	\end{equation}
The same holds true if we regard $x = x(\theta)$ as a function on $\T$. 
\begin{proof}
In the whole proof, we focus on the first sign choice in \eqref{eq:aacoord} and \eqref{eq:pertLemm}, the second one is completely analogous and hence omitted. 

The time-independent Hamiltonian \eqref{eq:pertHamil} is a conserved quantity along the Hamiltonian flow. By  restricting to the neighborhood of an isoenergy manifold with $E > 0$, which is topologically $\T$ and puts us in the rotating phase of the system \eqref{eq:pertHamil}, the local action-angle coordinates can be found in the following way. The action coordinate is obtained by integrating the momentum $p$ as a solution of
\begin{equation*}
\frac{p^2}{2} - \mu V(x) = E
\end{equation*} 
over a full rotation, i.e. 
\begin{equation*}
I_\mu =   \int_{0}^{1} \sqrt{2(E+ \mu V(x))} \, \D x\,.
\end{equation*}
This quantity is preserved along the Hamiltonian flow and one can express $E$ in terms of $I$ by the implicit function theorem. This allows us to calculate the time-derivative of the conjugate coordinate $\theta$ of $I$ as 
\begin{equation} \label{eq:frequency}
\dot{\theta} = \frac{\partial E}{\partial I} = \left(\frac{\partial I}{\partial E}\right)^{-1} = \left( \int_{0}^{1} \frac{\D x}{\sqrt{2(E+ \mu V(x))}} \right)^{-1} =: \omega\,.
\end{equation}
Integrating the equation of motion $\dot{x} = p = \sqrt{2(E+\mu V(x))}$ and using $\theta = \omega t$ (obtained by integrating \eqref{eq:frequency} w.r.t.~$t$ starting from $t=0$), we find $\theta = \theta(x)$ as a function of $x$ being given by 
\begin{equation*}
\theta = \omega \int_{0}^{x} \frac{\D x'}{\sqrt{2(E+\mu V(x'))}} = \frac{\int_{0}^{x} \frac{\D x'}{\sqrt{1+\mu V(x')/E}}}{\int_{0}^{1} \frac{\D x'}{\sqrt{1+\mu V(x')/E}}}\,.
\end{equation*}
From this, the approximation \eqref{eq:pertLemm} can now easily be derived by expanding the square roots using $\sqrt{1+y} = 1+\frac{y}{2} + \mathcal{O}(y^2)$ and $\frac{1}{1+y} = 1-y + \mathcal{O}(y^2)$ for $|y| \to 0$. The reversed statement for $x = x(\theta)$ is a simple consequence.
This finishes the proof of Lemma~\ref{lem:pertLemm}.
\end{proof}
\end{lem}

\section{Action-angle coordinates and analyticity} \label{app:AApendulum}
This appendix is concerned with analyticity properties of action-angle coordinates for one-dimensional Hamiltonian system
\begin{equation} \label{eq:pendHamil}
	H_\mu(p,x) = \frac{p^2}{2} - \mu V(x) 
\end{equation}
being defined on the cotangent bundle $T^*\T$, where $\mu$ is a positive parameter and $V\ge 0$ an analytic function. 
 Just as in Appendix \ref{app:pertLemm}, in the neighborhood of a fixed energy $E > 0$, we can find {\rm action-angle coordinates} $(I, \theta)$ of \eqref{eq:pendHamil} as
\begin{equation} \label{eq:aacoordpend}
	I = \pm \int_{0}^{1} \sqrt{2(E+ \mu V(x))} \, \D x\,, \qquad \theta = \pm  \frac{\int_{0}^{x} \frac{\D x'}{\sqrt{1+\mu V(x')/E}}}{\int_{0}^{1} \frac{\D x'}{\sqrt{1+\mu V(x')/E}}}\,. 
\end{equation}
 From now on, we shall restrict to the first sign choice in~\eqref{eq:aacoordpend}. 

 In our proofs of the analyticity cases in Theorem \ref{thm:2} and Theorem \ref{thm:3}, we shall exploit the fact that the function 
 \begin{equation} \label{eq:joint analytic}
\theta: (x, \mu) \mapsto \frac{\int_{0}^{x} \frac{\D x'}{\sqrt{1+\mu V(x')/E}}}{\int_{0}^{1} \frac{\D x'}{\sqrt{1+\mu V(x')/E}}}
 \end{equation}
is analytic in both variables. (Note that the further implicit dependence on $\mu$ via $E = E(I)$ is also analytic.) Now, for every fixed $\mu > 0$, the function $x \mapsto \theta(x, \mu)$ is analytic and invertible, and we denote its analytic inverse by $\theta \mapsto x_\mu(\theta)$ (cf.~\textbf{Step (i)} in the proofs of Theorem \ref{thm:2} and Theorem~\ref{thm:3}). Moreover, most importantly, also the function 
\begin{equation*}
(\theta, \mu) \mapsto  x_\mu(\theta)
\end{equation*}
is analytic in $\mu$, as shown in the following simple lemma applied to $f(z,w) \equiv \theta(x, \mu)$ in \eqref{eq:joint analytic}. 
\begin{lem}
Let $D_z, D_w \subset \R$ be open sets and 
\begin{equation} \label{eq:f anal}
f:D_z \times D_w \to \R\,, \quad (z,w) \mapsto f(z,w)
\end{equation}
an analytic function. Moreover, assume that the one-variable restriction $f(\cdot, w) : D_z \to \R$ is invertible and satisfies $f(D_z, w) = D$ for every fixed $w \in D_w$ and some open $D \subset \R$, such that we can write its analytic inverse function as
\begin{equation*}
f^{-1}(\cdot, w) : D \to D_z\,, \quad \zeta \mapsto f^{-1}(\zeta, w)\,. 
\end{equation*}
Then it holds that, with a slight abuse of notation, also 
\begin{equation*}
f^{-1} : D \times D_w \to D_z\,, \quad (\zeta,w) \mapsto f^{-1}(\zeta, w)
\end{equation*}
is an analytic function. 
\end{lem}
\begin{proof}
Since $f^{-1}(\cdot, w) : D \to D_z$ is analytic for every fixed $w \in D_w$, it can be represented as
\begin{equation} \label{eq:Cauchy}
f^{-1}(\zeta, w) = \frac{1}{2 \pi \I} \oint_C\frac{z \cdot (\partial_z f)(z,w)}{f(z,w)-\zeta} \D z
\end{equation}
by Cauchy's integral formula, where $C$ is a closed contour for which $|f(z,w) -\zeta| \ge \rho > 0$. In this form, since $f$ from \eqref{eq:f anal} is itself analytic and by involving Hartogs's theorem, the rhs.~of \eqref{eq:Cauchy} defines (locally) an jointly analytic function in both variables $(\zeta,w)$. 
\end{proof} 
 We note that, although $\theta$ from \eqref{eq:joint analytic} is always analytic in $\mu$, the lower regularity in $x$ for a general $V \in C^2(\T)$ prevents the analyticity in $\mu$ to carry over to the inverse function.

 We conclude this appendix, by showing analyticity for the important special case of a pendulum, i.e.~$V(x) = 1 - \cos(2 \pi x)$, in a more explicit way. In this particular situation, $\theta$ can be represented as
\begin{equation} \label{eq:theta pend}
	\theta = \frac{1}{2} - \frac{F\big( \pi (\tfrac{1}{2} - x)\,|\,m_\mu \big)}{2 K(m_\mu)}\,,
\end{equation}
where we introduced the shorthand notation $m_\mu = \tfrac{2 \mu}{E + 2 \mu}$. 
Here, $F(\varphi \,|\, m)$ (resp.~$K(m)$) for $k \in [0,1)$ denotes the \emph{incomplete (resp.~complete) elliptic integral of the first kind}, i.e.~
\begin{equation} \label{eq:elliptic int}
	F(\varphi \,|\,  m) = \int_{0}^{\varphi} \frac{\D \vartheta}{\sqrt{1 - m \sin^2(\vartheta)}} \qquad \text{and} \qquad K(m) = F(\tfrac{\pi}{2}\,|\, m)\,. 
\end{equation}
The quantity $m$ is called the \emph{parameter}, $\varphi$ the \emph{amplitude}.

Now, the so-called \emph{Jacobi elliptic function} are obtained by inverting the incomplete elliptic integral \eqref{eq:elliptic int}. More precisely, if	$u := F(\varphi\,|\,m)$ denotes the \emph{argument}, and $u$ and $\varphi$ are related in this way (we also write $\varphi = \mathrm{am}(u\,|\,m)$ for the amplitude), then we define the \emph{Jacobi elliptic functions} as
\begin{equation*} 
	\mathrm{sn}(u\,|\,m) := \sin\big( \mathrm{am}(u\,|\,m) \big)\,, \qquad \mathrm{cn}(u\,|\,m) := \cos\big( \mathrm{am}(u\,|\,m) \big)\,, 
\end{equation*}
which are called the \emph{elliptic sine} and \emph{elliptic cosine}, respectively. Moreover, using the notation introduced above, we can invert the relation \eqref{eq:theta pend} to find that
\begin{equation} \label{eq:x pend}
	x = \frac{1}{2} - \frac{1}{\pi} \, \mathrm{am}\big(K(m_\mu) (1 - 2 \theta_\mu) \,|\,  m_\mu \big)\,.
\end{equation}

Most commonly, the elliptic sine and cosine are considered for fixed parameter $m$ as functions of $u$, in which way they in fact behave as \emph{elliptic functions}, i.e.~doubly-periodic meromorphic function on the complex plane. However, as a function of the parameter parameter $m$ (see \cite{Walker}), we have that
\begin{equation} \label{eq:analytic}
m \mapsto \mathrm{sn}(K(m)u\,|\,m) \quad \text{and} \quad m \mapsto \mathrm{cn}(K(m)u\,|\,m)
\end{equation}
are analytic for $m \in \C \setminus [1,\infty)$ and fixed $u \in \R$. This easily follows by representing $\mathrm{sn}(K(m)u\,|\,m)$ and $\mathrm{sn}(K(m)u\,|\,m)$ as ratios of \emph{Jacobi theta functions} \cite[Eq.~16.36.3]{AbraSteg} (see also \cite[Eq.~5]{SieglStampach}), whose zeros are known explicitly \cite[Eq.~16.36.2]{AbraSteg}.

\section{Weak KAM theory} \label{app:weakKAM} 
In this appendix, we provide an overview on basic concepts and results of weak KAM theory and Aubry-Mather theory, which are relevant in the proofs of our main results. In particular, we discuss separable Hamiltonian systems on $T^*\T^2$, i.e.~sums of two independent systems on $T^*\T$. The presentation partly follows lecture notes from Sorrentino \cite{SorrentinoLecNotes}, which build on seminal works from Mather \cite{Mather0, Mather1, Mather2}, Aubry \cite{Aubry}, Ma\~né \cite{mane}, Fathi \cite{fathiorig, fathi} and others. 

In the following, let $(M,g)$ be a compact and connected smooth Riemannian manifold without boundary, e.g.~the torus $\T^2$. As in Section~\ref{subsec:geoflow}, $TM$ denotes its tangent bundle and $T^*M$ its cotangent bundle. While a point in $TM$ is denoted by $(x,v)$, where $x \in M$ and $v \in T_xM$, a point in $T^*M$ is denoted by $(x,p)$, where $x \in M$ and $p \in T^*_xM$ is a linear functional on $T_xM$. The Riemannian metric $g$ induces a metric $d$ on $M$ as well as a norm $\Vert \cdot \Vert_x$ on $T_xM$. We shall use the same notation for the norm induced on $T^*_xM$. The standard assumptions on a Hamiltonian $H: T^*M \to \R$ are summarized as follows. 
\begin{defi} {\rm (Tonelli Hamiltonians)} \label{def:Tonelli}\\
A function $H : T^*M \to \R$ is called a {\it Tonelli Hamiltonian} if and only if $H$ is (i) of class $C^2$; (ii) strictly convex in each fiber in $C^2$-sense, i.e.~the quadratic form $(\partial^2H/\partial p^2)(x,p)$ is positive definite for any $(x,p) \in T^*M$; (iii) superlinear in each fiber, i.e.~$\lim\limits_{\Vert p \Vert_x \to \infty} \frac{H(x,p)}{\Vert p \Vert_x} = \infty$. 
\end{defi}
A Hamiltonian $H:T^*M \to \R$ is canonically associated to a Lagrangian $L:TM \to \R$ as being each others Fenchel-Legendre transforms $(x,p) = \mathcal{L}(x,v)$, i.e.
\begin{alignat*}{3}
H : & \ T^*M \to \R, \quad &&(x,p) \mapsto \sup_{v \in T_xM} &&\big[ \langle p, v \rangle_x - L(x,v)   \big]\,, \\
L : &\ TM \to \R, \quad &&(x,v) \mapsto \sup_{p \in T^*_xM} &&\big[ \langle p, v \rangle_x - H(x,p)   \big]\,.
\end{alignat*}
It is easy to check that the Lagrangian associated to a Tonelli Hamiltonian is also of Tonelli type (defined analogously to Definition \ref{def:Tonelli}). 

Piecewise $C^1$ curves $\gamma:[0,1] \to M$, which minimize the action functional 
\begin{equation*}
A_L(\gamma) := \int_{0}^{1} L(\gamma(t), \dot{\gamma}(t)) \D t\,,
\end{equation*}
satisfy the associated Euler-Lagrange equation
\begin{equation*}
\frac{\D}{\D t} \frac{\partial L}{\partial v} (\gamma(t), \dot{\gamma}(t)) = \frac{\partial L}{\partial x} (\gamma(t), \dot{\gamma}(t))\,, \qquad t \in [0,1]\,.
\end{equation*}
In case that $\det \frac{\partial^2L}{\partial v^2} \neq 0$ (Legendre condition), the Euler-Lagrange equation is equivalent to an equation, which can be solved for $\ddot{\gamma}(t)$, which allows to define a vector field $X_L$ on $TM$, such that the solutions of $\ddot{\gamma}(t) = X_L(\gamma(t), \dot{\gamma}(t))$ precisely satisfy the Euler-Lagrange equation. The associated flow $\Phi_t^L$ is called the \emph{Euler-Lagrange flow}, which is $C^1$ for $L$ of class $C^2$. 
\subsection{Basic notions of Aubry-Mather theory} \label{subsec:AubryMather}
The central objects of study in Aubry-Mather theory are invariant probability measures on $TM$ having finite average action, 
\begin{equation*} 
\mathfrak{M}(L):= \left\{  \nu \ \text{prob. meas. on} \ TM \ \text{with} \ \nu \circ \Phi_t^L = \nu\,, \, \forall t \in \R\,, \ \int_{TM} L \, \D \nu < \infty  \right\}\,,
\end{equation*}
which shall be endowed with the vague topology, i.e.~the weak$^*$ topology induced by the continuous functions $f:TM \to \R$ having at most linear growth. 
\begin{prop}
Every non-empty energy level $\{ H \circ \mathcal{L}(x,v) = E \}$ contains at least one invariant probability measure of $\Phi_t^L$, i.e.~$\mathfrak{M}(L) \neq \emptyset$. 
\end{prop}
For every $\nu \in \mathfrak{M}(L)$ we now define its average action
\begin{equation*}
A_L(\nu) := \int_{TM} L \D \nu\,.
\end{equation*}
Since $A_L: \mathfrak{M}(L) \to \R$ is lower semicontinuous w.r.t.~the vague topology on $\mathfrak{M}(L)$, we have the following.
\begin{prop}
There exists $\nu \in \mathfrak{M}(L)$, which minimizes $A_L$ over $\mathfrak{M}(L)$. 
\end{prop}
A measure $\nu \in \mathfrak{M}(L)$ minimizing $A_L$ is called an action-minimizing measure of $L$. The principal goal in Aubry-Mather theory is to characterize invariant sets of the dynamics via action minimizing measures. Since -- at least for integrable systems -- the phase space is foliated by invariant tori (cf.~Theorem \ref{thm:liouvillearnold}) and minimizing a single functional will not be sufficient to characterize all of them, one considers certain modifications of the Lagrangian: Let $\eta$ be a $1$-form on $M$ and interpret it as a functional on the tangent space as
\begin{equation*}
\hat{\eta} : TM \to \R\,, \ (x,v) \mapsto \langle \eta(x), v\rangle_x\,.
\end{equation*}
One can easily verify that, if $\eta$ is closed (i.e.~$\D \eta = 0$), then $L$ and $L_\eta := L - \hat{\eta}$ have the same Euler-Lagrange flow. Moreover, if $\eta = \D f$ is an exact $1$-form, then $\int \widehat{\D f} \D \nu = 0$ for any $\nu \in \mathfrak{M}(L)$. Therefore, for fixed $L$, the minimizing measures of $L_\eta$ will depend only on the de Rham cohomology class $\cC = [\eta] \in H^1(M, \R)$. Hereafter, $\eta_c$ shall denote a closed $1$-form with cohomology class $\cC \in H^1(M,\R)$. 

\begin{defi}We define Mather measures, Mather's $\alpha$-function, and Mather sets as follows:
	\begin{itemize}
\item 	If $\nu \in \mathfrak{M}(L)$ minimizes $A_{L_{\eta_\cC}}$, we call $\nu$ a {\it Mather measure with cohomology~$\cC$}. 
\item  The map
\begin{equation} \label{eq:alphafunc}
\alpha : H^1(M,\R) \to \R\,, \ \cC \mapsto - \min_{\nu \in \mathfrak{M}(L)} A_{L_{\eta_\cC}} (\nu)
\end{equation}
is called \textit{Mather's $\alpha$-function}. It is well defined and easily seen to be convex. 
\item For $\cC \in H^1(M,\R)$, we define the \textit{Mather set of cohomology class $\cC$ as}
\begin{equation*} 
\widetilde{\mathcal{M}}_\cC := \bigcup_{\nu \in \mathfrak{M}_\cC(L)} \mathrm{supp} \, \nu \subset TM\,,
\end{equation*}
where we denoted $\mathfrak{M}_\cC(L) := \{ \nu \in \mathfrak{M}_L \ : \ A_{L_{\eta_\cC}}(\nu) = - \alpha(\cC) \}$. The projection $\mathcal{M}_\cC = \pi \big(  \widetilde{\mathcal{M}}_\cC\big) \subset M$ on the base manifold is called the \textit{projected Mather set with cohomology class $\cC$}. 
	\end{itemize}
\end{defi}
By duality, one can also define Mather's $\beta$-function: Since $\int_{TM} \hat{\eta} \D \nu = 0$ for exact $1$-forms $\eta$, the linear functional
\begin{equation*}
H^1(M, \R) \to \R\,, \ \cC \mapsto \int_{TM}\hat{\eta_\cC} \D \nu
\end{equation*} 
is well defined. By duality, there exists $\boldsymbol{\omega}(\nu) \in H_1(M,\R)$ such that 
\begin{equation*}
\int_{TM} \hat{\eta}_{\cC} \D \nu = \langle \cC, \boldsymbol{\omega}(\nu) \rangle \qquad \forall \cC \in H^1(M,\R)\,,
\end{equation*}
where we call $\boldsymbol{\omega}(\nu)$ the rotation vector of $\nu$, which will turn to out to be matching the earlier definition in Proposition \ref{prop:KAM}. The map $\boldsymbol{\omega}: \mathfrak{M}_L \to H_1(M, \R)$ is continuous, affine linear and surjective. In combination with the lower semicontinuity of $A_L$, this shows that Mather's $\beta$-function
\begin{equation} \label{eq:betafunc}
\beta : H_1(M, \R) \to \R\,, \ \boldsymbol{h} \mapsto \min_{\nu \in \mathfrak{M}(L) : \boldsymbol{\omega}(\nu) = \boldsymbol{h}} A_L(\nu)
\end{equation}
is well defined. It can easily seen to be the convex and, in fact, being the convex conjugate (Fenchel transform) of the $\alpha$-function, showing that both, $\alpha$ and $\beta$, have superlinear growth. 

We will see below, that the Liouville tori $T_\cC$ with $|c_i| > \sqrt{\mu_i} \mathfrak{c}(V_i)$ from Proposition~\ref{prop:KAM} agree with the Mather set of cohomology class $\cC \in H^1(\T^2, \R) \cong \R^2$, i.e.~$\widetilde{\mathcal{M}}_\cC = \mathcal{L}^{-1}(T_\cC)$. Basically, this will be concluded from the following two fundamental results. 
\begin{thm} {\rm (Mather's graph Theorem \cite{Mather1})} \label{thm:mathergraph}\\
The Mather set $\widetilde{\mathcal{M}}_\cC$ is compact, invariant under the Euler-Lagrange flow and $\pi \vert_{\widetilde{\mathcal{M}}_\cC}: \widetilde{\mathcal{M}}_\cC \to M$ is an injective map, whose inverse $\pi^{-1} : \mathcal{M}_\cC \to \widetilde{\mathcal{M}}_\cC$ is Lipschitz.  
\end{thm}
\begin{thm} {\rm (Carneiro \cite{carneiro})} \label{thm:carneiro}\\
The Mather set $\widetilde{\mathcal{M}}_\cC$ is contained in the energy level $\{ H \circ \mathcal{L}(x,v) = \alpha(\cC) \}$. 
\end{thm}
\subsection{Aubry-Mather theory in one dimension}
In the following, we discuss the basic objects introduced above for the one-dimensional example of a mechanical Hamiltonian on $M = \T$. Note that the unperturbed Hamiltonian~\eqref{eq:H0} in the formulation of our main results is a sum of two such one-dimensional systems. 
Let $V \in C^2(\T)$ be a non-negative Morse function with ${\min_{x \in \T} V(x) = 0}$,
$\mu \in (0,1]$, and consider the Hamiltonian 
\begin{equation} \label{eq:1Dsys}
H: T^*\T \to \R\,, \ (x,p) \mapsto \frac{p^2}{2} - \mu V(x)\,,
\end{equation}
whose corresponding Lagrangian can easily be obtained as $L(x,v) = \frac{v^2}{2} + \mu V(x)$. We note that 
\begin{equation*} 
T\T \cong T^*\T \cong \T \times \R \qquad \text{and} \qquad H_1(\T, \R) \cong H^1(\T,\R) \cong \R\,.
\end{equation*}
First of all, we study invariant probability measures of the system \eqref{eq:1Dsys}. 
\begin{itemize}
\item Since $V$ is a Morse functions, the sets of local (isolated) minima and maxima, $\mathfrak{X}_{\min} $ and $\mathfrak{X}_{\max}$, respectively, contain only finitely many elements. This shows that each of the measures
\begin{equation*}
\big(\delta_{(x_*, 0)}\big)_{x_* \in \mathfrak{X}_{\min}}\,, \qquad \big(\delta_{(x^*, 0)}\big)_{x^* \in \mathfrak{X}_{\max}}
\end{equation*}
are invariant probability measures of the system, all having zero rotation vector. They correspond to unstable and stable fixed points with respective energies $H(x_*, 0) = - \mu V(x_*)$ for $x_* \in \mathfrak{X}_{\min}$ and $H(x^*, 0) = -\mu V(x^*)$ for $x^* \in \mathfrak{X}_{\max}$. 
\item For $E  > 0$, the energy level $\{ H(x,p) = E \}$ consists of two homotopically non-trivial periodic orbits
\begin{equation*}
\mathcal{P}^\pm_E := \left\{ (x,p) \, : \, p = \pm \sqrt{2(E + \mu V(x))}\,, \ x \in \T  \right\}\,.
\end{equation*}
The probability measures evenly distributed along these orbits -- denoted by $\nu_E^\pm$ -- are invariant probability measures of the system. If we denote by
\begin{equation*}
T(E) := \int_{0}^{1} \frac{1}{\sqrt{2(E + \mu V(x))}} \D x
\end{equation*}
the period of such an orbit, one can easily see that $\boldsymbol{\omega}(\nu_E^\pm) = \pm \frac{1}{T(E)}$. Moreover, we have that $T:(0,\infty) \to (0,\infty)$ is continuous, strictly decreasing, and $T(E) \to \infty$ as $E \to 0$, i.e.~$\boldsymbol{\omega}(\nu_E^\pm) \to 0$ as $E \to 0$. 
\item For every $E \in (-\mu \max_{x \in \T} V(x), 0) \setminus ((- \mu V(\mathfrak{X}_{\max})) \cup (- \mu V(\mathfrak{X}_{\min})))$, the energy level $\{H(x,p) = E \}$ consists of finitely many disjoint \emph{contractible} periodic orbits. A probability measure $\nu_E^{(k)}$, $k \in \{1, ... , N_E\}$, evenly distributed along such an orbit, is invariant for the system. Since the orbit is contractible, the rotation vector of $\nu_E^{(k)}$ is zero, $\boldsymbol{\omega}(\nu_E^{(k)}) = 0$. 
\end{itemize}
 The support of the measures $\nu_E^{(k)}$ for $E \in (-\mu \max_{x \in \T}V(x), 0) \setminus ((- \mu V(\mathfrak{X}_{\max})) \cup (- \mu V(\mathfrak{X}_{\min})))$ is not a graph over $\T$. Therefore, by means of Mather's graph Theorem~\ref{thm:mathergraph}, they cannot be action minimizing. Moreover, we also have that the $\alpha$-function is even, $\alpha(c) = \alpha(-c)$ for all $c \in \R$, which follows by the symmetry $H(x,p) = H(x,-p)$ of the system \eqref{eq:1Dsys}. In combination with the convexity of $\alpha$, this shows that $\min_\R \alpha(c) = \alpha (0)$. Since $V \ge 0$, we have $A_L(\nu) \ge 0$ for  all $\nu \in \mathfrak{M}_L$ and thus $\alpha(c) \ge 0 $ for all $c \in \R$. By taking $x_* \in \mathfrak{X}_{\min}$ with $V(x_*) = 0$ (a global minimum), we have $A_L(\delta_{(x_*, 0)}) = 0$, which shows that $\min_\R \alpha(c) = \alpha(0) =  0$. It follows from Theorem \ref{thm:carneiro}, that only energy levels with $E \ge 0$ are capable of containing a Mather set. The Mather set of cohomology $c = 0$ is contained in the energy level with $E = 0$ and we have $\widetilde{\mathcal{M}}_0 = \set{V = 0} \times \set{0}$.
 
For cohomology classes different from zero, a first observation is that, since $\alpha$ is superlinear and continuous, all energy levels with $E > 0$ must contain some Mather set. Let $E > 0$ and consider the periodic orbit $\mathcal{P}_E^+$ with the invariant probability measure $\nu_E^+$ evenly distributed. The graph of this orbit can be viewed as the graph of the closed $1$-form $\eta_E^+ := \sqrt{2(E + \mu V(x))} \D x$, having cohomology class
 \begin{equation*}
c^+(E) = [\eta_E^+] = \int_{0}^{1} \sqrt{2(E + \mu V(x))} \D x\,.
 \end{equation*}
 This function is continuous, strictly increasing for $E > 0$ and we have 
 \begin{equation} \label{eq:defc(V)}
 c^+(E) \longrightarrow \sqrt{\mu}\int_{0}^{1} \sqrt{2 V(x)}\D x =: \sqrt{\mu} \, \mathfrak{c}(V)\,, \quad \text{as} \quad E \to 0\,.
 \end{equation}
 Therefore, this defines an invertible function $c^+:(0,\infty) \to (\sqrt{\mu} \, \mathfrak{c}(V), \infty)$, whose inverse we denote by $E(c)$. Using Mather's graph Theorem \ref{thm:mathergraph} and the Fenchel-Legendre inequality, $\langle v,p \rangle_x \le L(x,v) + H(x,p)$, one can show that $\nu_E^+$ is the unique $c^+(E)$-action minimizing measure and we thus have
 \begin{equation*}
\widetilde{\mathcal{M}}_{c^+(E)} = \mathcal{P}_E^+\,, \quad \text{and similarly} \quad \widetilde{\mathcal{M}}_{c^-(E)} = \mathcal{P}_E^-\,,
 \end{equation*}
 where $c^-(E) = -c^+(E)$ is the cohomology class of $\eta_E^- = - \eta_E^+$. 
 
 It remains to study the non-zero cohomology classes in $[- \sqrt{\mu} \mathfrak{c}(V), \sqrt{\mu} \mathfrak{c}(V)]$. Observe that, by Theorem \ref{thm:carneiro}, we have $\alpha(c^\pm(E)) = E$ and thus, using continuity of $\alpha$, it follows that $\alpha(\pm \sqrt{\mu} \mathfrak{c}(V)) = 0$. By convexity of $\alpha$ and $\min_\R\alpha(c) = \alpha(0) = 0$, this implies $\alpha (c) \equiv 0$ for $c \in [- \sqrt{\mu} \mathfrak{c}(V), \sqrt{\mu} \mathfrak{c}(V)]$. Consequently, the corresponding Mather measures lie in the zero energy level, such that we have 
 \begin{equation*}
\widetilde{\mathcal{M}}_c = \set{V = 0} \times \set{0} \qquad \text{for all} \quad c \in [- \sqrt{\mu} \mathfrak{c}(V), \sqrt{\mu} \mathfrak{c}(V)]\,.
 \end{equation*}
Summarizing the above considerations, we have shown that 
\begin{equation} \label{eq:Mathersetandfunction}
\widetilde{\mathcal{M}}_c = \begin{cases}
\{V = 0\} \times \{ 0 \} \hspace{1.5mm} &\text{if} \ |c| \le  \sqrt{\mu} \mathfrak{c}(V)\\
\mathcal{P}_{E(|c|)}^{\mathrm{sgn}(c)} \hspace{1.5mm} &\text{if} \ |c| > \sqrt{\mu} \mathfrak{c}(V)
\end{cases}\,, \quad 
\alpha(c) = \begin{cases}
0 \hspace{1.5mm} &\text{if} \ |c| \le  \sqrt{\mu} \mathfrak{c}(V)\\
E(|c|) \hspace{1.5mm} &\text{if} \ |c| > \sqrt{\mu} \mathfrak{c}(V)
\end{cases}\,. 
\end{equation}
\begin{rmk}\label{rmk:reg and mon of alpha}
{ We note that $\alpha$ from \eqref{eq:Mathersetandfunction} is globally $C^1$ (which follows from strict convexity of its Fenchel transform $\beta$ defined in \eqref{eq:betafunc}) and $C^\infty$, even analytic, for $|c| > \sqrt{\mu} \mathfrak{c}(V)$ (which follows from the implicit function theorem as $E (\cdot)= (c^+)^{-1}(\cdot)$). Also, $\alpha$ is symmetric around $0$ and (strictly) increasing for $c \ge 0$ (for $c \ge \sqrt{\mu}\mathfrak{c}(V)$). }
\end{rmk}
\begin{rmk} \label{rmk:propertyP1}
By arguing as above for the two independent dimensions of \eqref{eq:H0}, this demonstrates the connection between part (a) of Theorem \ref{thm:liouvillearnold} and part (a) of Proposition~\ref{prop:KAM}. More precisely, the graph property follows from Theorem~\ref{thm:carneiro} and the results in \eqref{eq:Mathersetandfunction}. The remaining part of the statement follows after realizing that $\alpha(\cC) = \alpha_1(c_1) + \alpha_2(c_2)$, where $\alpha_i$ is the $\alpha$-function of the one-dimensional system with coordinates labeled by $i$, and taking $u_\cC \in C^3(\T^2)$ with $|c_i| > \sqrt{\mu_i} \mathfrak{c}(V_i)$  according to 
\begin{equation*}
\nabla_x u_\cC(x) = -\cC \pm \left(\begin{matrix}
\sqrt{2(\alpha_1(c_1) + \mu_1V_1(x^1))} \\ \sqrt{2(\alpha_2(c_2) + \mu_2 V_2(x^2))}
\end{matrix}\right)\,,
\end{equation*}
(recall $V_i \in C^2(\T)$ is a non-negative Morse function and $\alpha_i(c_i) > 0$) such that the Hamilton-Jacobi equation 
\begin{equation*}
\alpha_{}(\cC) = H_0(x, \cC + \nabla_x u_{ \cC}(x))
\end{equation*}
is satisfied.
Moreover, in case that $U$ as in \eqref{eq:Heps} \emph{is} actually separable, one can employ the explicit forms for $c^+(E)$ as the inverse of the $\alpha$-function and $\nabla u_\cC$ to \emph{prove} the validity of Assumption~\eqref{itm:P}, simply by using the same expansions leading to the proof of Lemma \ref{lem:pertLemm}. This means that separable systems satisfy Assumption~\eqref{itm:P}, which shows consistency with our main results.
\end{rmk}
\subsection{Fathi's weak KAM theory and perturbations} \label{subsec:Fathiweak}
For concreteness, we specialize to $M = \T^2$, in which case $H^1(\T^2, \R) \cong T^*_x\T^2 \cong \R^2$ for every $x \in \T^2$, such that we can identify $\cC \in H^1(\T^2, \R)$ with a closed $1$-form of cohomology class $\cC$. The central object of investigation in Fathi's weak KAM theory \cite{fathiorig, fathi} (with important contributions from Siconolfi \cite{fathisic, fathisico05}, Bernard \cite{bernard} and others \cite{crandalllions, LionsPapaVara}) is the Hamilton-Jacobi equation (HJE)
\begin{equation} \label{eq:HJEapp}
H(x, \cC + \nabla_x u) = k\,, \quad k \in \R\,,
\end{equation}
where $H$ is a Tonelli Hamiltonian on $T^*\T^2$ with associated Tonelli Lagrangian $L$. 

For classical solutions, i.e.~$C^1$-functions $u:\T^2 \to \R$ solving \eqref{eq:HJEapp}, it is immediate to check, that there is at most one value $k \in \R$, for which such a $C^1$-solution may exist. In fact, this value agrees with Mather's $\alpha$-function defined in \eqref{eq:alphafunc}. The primary goal of the theory is to define a weaker notion of (sub)solution (so called \emph{weak KAM solutions}), whose existence is always guaranteed \cite{fathi, fathiorig}, even if the Tonelli Hamiltonian $H$ is \emph{not} integrable. See \cite{crandalllions, LionsPapaVara, fathisico05} for approaches to the problem from the theory of partial differential equations.

The following proposition contains perturbative properties of weak KAM solutions $u_\epsi$ and Mather's $\alpha$-function $\alpha_\epsi$ for systems of the form 
\begin{equation*}
H_\epsi(x,p) = H_0(x,p) + \epsi H_1(x,p)\,.
\end{equation*}
\begin{prop}{\rm (Gomes \cite{gomes})} \label{prop:gomes}\\
Let $H_0:T^*\T^2 \to \R$ be an integrable Hamiltonian and $u^{(0)}$ a (classical) $C^1$-solution of the HJE $H_0(x, \cC + \nabla_x u^{(0)}) = \alpha^{(0)}(\cC)$. Moreover, let $\nu^{(0)}$ denote the projection of a Mather measure with cohomology $\cC$. Suppose there exists a function $u^{(1)} \in C^1(\T^2)$ and a number $\alpha^{(1)}(\cC)$ such that 
\begin{equation} \label{eq:firstorderaltern}
\alpha^{(1)}(\cC) = \langle (\nabla_p H_0)(x, \cC + \nabla_x u^{(0)}), \nabla_x u^{(1)} \rangle + H_1(x, \cC + \nabla_x u^{(0)})\,, \quad \forall x \in \T^2\,.
\end{equation}
Then 
\begin{equation} \label{eq:alpha1}
\alpha^{(1)}(\cC) = \int_{\T^2} H_1(x, \cC + \nabla_x u^{(0)}) \, \D \nu^{(0)} \quad \text{and} \quad \alpha_\epsi(\cC) = \alpha^{(0)}(\cC) + \epsi \alpha^{(1)}(\cC) + \mathcal{O}_\cC(\epsi^2)\,.
\end{equation}
\end{prop}

\begin{rmk} \label{rmk:alterP(iii)}
The above proposition provides a converse to \eqref{eq:propertyP1} in Assumption \eqref{itm:P}. In fact, the transport-type equation \eqref{eq:firstorderaltern} for the unknown $u^{(1)}$ (with so far unspecified constant $\alpha^{(1)}(\cC)$) is exactly the first-order expansion obtained in \eqref{eq:firstorder1}, \eqref{eq:firstorder2}, and \eqref{eq:firstorder3} in Section~\ref{sec:proofs} and also fixes $\alpha^{(1)}(\cC)$ to be given by \eqref{eq:alpha1}. Moreover, the equation \eqref{eq:firstorderaltern} coincides with the relation, which the correction term $u^{(1)}$ of an \emph{approximate solution} $\tilde{u}_\epsi = u^{(0)} + \epsi \,  u^{(1)}$ to the HJE
\begin{equation*}
H_\epsi(x, \cC + \nabla_x u_\epsi) = k
\end{equation*}
of order one has to satisfy (see \cite{gomes}). The approximate solution $\tilde{u}_\epsi = u^{(0)} + \epsi \,  u^{(1)}$ also coincides with the first order truncation of the so-called \emph{Lindstedt series} \cite{AKN, gomes2}, a not necessarily convergent perturbative expansion similar to the ones in KAM theory \cite{KAM_K, KAM_A, KAM_M} or the Poincaré-Melnikov method \cite{AKN, Guckenheimer, TreschZube}. Finally, it is interesting to note that, if $H_1(x,p) = W(x)$ is independent of the $x$-variables, then $\alpha_\epsi(\cC)$ is a convex function of $\epsi$ and thus a.e.~twice differentiable -- yielding the expansion \eqref{eq:alpha1} at every such point. 
\end{rmk}

\end{document}